\documentclass[11pt,letterpaper,reqno]{amsart}
\usepackage[T1]{fontenc}
\usepackage{lmodern}
\usepackage{amsmath,amssymb,amsthm,mathtools,bm}
\usepackage[margin=1in]{geometry}
\usepackage{microtype}
\usepackage{enumitem}
\usepackage{hyperref}
\usepackage{bookmark}
\usepackage{doi}
\hypersetup{bookmarksdepth=2,hidelinks,pdftitle={Holomorphic curves of finite lower order with few inflection points},
 pdfauthor={Alexandre Eremenko and Teng Zhang}}
\setlist[enumerate]{label=\textup{(\roman*)},leftmargin=*,itemsep=2pt,topsep=4pt}
\setlist[itemize]{leftmargin=*,itemsep=2pt,topsep=4pt}
\allowdisplaybreaks[1]
\newtheorem{thm}{Theorem}[section]
\newtheorem{lem}[thm]{Lemma}
\newtheorem{prop}[thm]{Proposition}
\newtheorem{cor}[thm]{Corollary}
\newtheorem{conj}[thm]{Conjecture}
\newtheorem{thmA}{Theorem}

\theoremstyle{remark}

\numberwithin{equation}{section}
\newcommand{\C}{\mathbb C}
\newcommand{\R}{\mathbb R}
\newcommand{\Z}{\mathbb Z}
\newcommand{\PP}{\mathbb P}
\newcommand{\dd}{\,\mathrm d}
\newcommand{\dA}{\,\mathrm dA}
\newcommand{\loc}{\mathrm{loc}}
\newcommand{\TN}{T_{\mathrm N}}
\newcommand{\mean}[2]{\mathcal M_{#1}\!\left[#2\right]}
\newcommand{\norm}[1]{\left\lVert#1\right\rVert}
\newcommand{\convmeas}{\xrightarrow{\,\mathrm{meas}\,}}
\newcommand{\fv}{\bm f}
\newcommand{\gv}{\bm g}
\newcommand{\pv}{\bm p}
\newcommand{\hv}{\bm h}
\newcommand{\yv}{\bm y}
\newcommand{\ev}{\bm e}
\newcommand{\Gn}{\mathcal R_{n+1}}
\DeclareMathOperator{\supp}{supp}
\DeclareMathOperator{\ord}{ord}
\DeclareMathOperator{\area}{area}

\DeclareMathOperator{\Ree}{Re}
\newcommand{\step}[1]{\medskip\noindent\emph{#1}\ }

\begin{document}
\title[Curves with few inflection points]{Holomorphic curves of finite lower order with few inflection points}
\author[A.~Eremenko]{Alexandre Eremenko}
\address{Mathematics Department, Purdue University, West Lafayette, IN 47907, USA}
\email{eremenko@purdue.edu}
\author[T.~Zhang]{Teng Zhang}
\address{School of Mathematics and Statistics, Xi'an Jiaotong University,
Xi'an 710049, P. R. China}
\email{teng.zhang@stu.xjtu.edu.cn}
\date{}
\begin{abstract}
We prove a conjecture posed by the first-named author in 1998. Let
$f\colon\C\to\PP^n$ be a transcendental linearly non-degenerate holomorphic
curve of finite lower order. If the counting function $N_1(r,f)$ of its
Wronskian zeros satisfies $N_1(r,f)=o(T(r,f))$, then its order and lower order
coincide and belong to
$\{1+k/q:k\in\Z_{\ge0},\ 2\le q\le n+1\}$, and its characteristic is
regularly varying. Every order in this set occurs. We also prove the sharp
inequality $\limsup_{r\to\infty}N_1(r,f)/T(r,f)\ge1$ for transcendental
linearly non-degenerate curves of order zero.
\end{abstract}
\maketitle
\tableofcontents

\section{Introduction}\label{sec:intro}

We use the standard terminology and facts of Nevanlinna theory in
one complex variable; see, for example, \cite{Hay64}. The main definitions
will be recalled below for holomorphic curves in projective space.
The lower order and order of a meromorphic function $f$ are defined by
\[
 \lambda:=\liminf_{r\to\infty}\frac{\log T(r,f)}{\log r}
 \leq\limsup_{r\to\infty}\frac{\log T(r,f)}{\log r}=: \rho.
\]
Thus $0\leq\lambda\leq\rho\leq+\infty$.

The simplest prototype of the results considered here is the following
classical elementary fact. If $f$ is a transcendental entire function of
finite lower order such that $f'(z)\ne0$ for every $z\in\C$, then
\[
 f(z)=\int_0^z\exp P(\zeta)\,\mathrm d\zeta+C,
\]
where $P$ is a nonconstant polynomial. In particular,
$
 \rho(f)=\lambda(f)=\deg P\geq1.
$

For meromorphic functions, the
analogous argument is somewhat harder. If $f$ is a transcendental
meromorphic function of finite lower order with no critical points, then
its Schwarzian derivative is entire, and the lemma on the
logarithmic derivative
shows that the Schwarzian derivative is a polynomial. Consequently $f=w_1/w_0$, where
$w_0,w_1$ are linearly independent entire solutions of
\[
 w''+Pw=0
\]
with a polynomial coefficient $P$. In the transcendental case $P\not\equiv0$.
Well-known asymptotic integration theory (see, for example,
\cite{Was65}) shows that the lower order and order coincide and are equal to
$(\deg P+2)/2\geq1$.
This observation is due to Frithiof Nevanlinna \cite{Nev30}, who conjectured
that finite order together with the weaker condition
$N_1(r)=o(T(r))$ on critical points suffices for the same conclusion about
regularity of growth. When discussing a fixed function, we omit it from
the notation when no confusion can arise.

In what follows we assume that $f$ is transcendental.
Recall the Second Fundamental Theorem of Nevanlinna:
\begin{equation}\label{1}
 \sum_{j=1}^q m(r,a_j)+N_1(r)\leq(2+o(1))T(r),\qquad r\to\infty.
\end{equation}
Here the $a_j$ are distinct points of $\PP^1=\C\cup\{\infty\}$,
$m(r,a_j)$ are proximity functions, $T(r)$ is the Nevanlinna
characteristic, and
\[
 N_1(r):=N(r,0,f')-N(r,f')+2N(r,f)
\]
is the counting function of critical points. For functions of finite order,
\eqref{1} holds as written; in general one excludes a set of $r$'s of
finite length.

The \emph{deficiency} is defined as
\[
 \delta(a,f)=\liminf_{r\to\infty}\frac{m(r,a)}{T(r)}
 =1-\limsup_{r\to\infty}\frac{N(r,a)}{T(r)}.
\]
The equality follows from the First Main Theorem, and
$0\leq\delta(a,f)\leq1$. For functions of finite order, \eqref{1} gives
the defect relation
\begin{equation}\label{2}
 \sum_{a\in\PP^1}\delta(a,f)
 +\limsup_{r\to\infty}\frac{N_1(r)}{T(r)}\leq2.
\end{equation}
In particular, the maximal sum of deficiencies condition
\begin{equation}\label{sum}
 \sum_a\delta(a,f)=2
\end{equation}
implies F.~Nevanlinna's small-ramification condition
\begin{equation}\label{smal}
 N_1(r)=o(T(r)).
\end{equation}

There is a substantial literature on the consequences of \eqref{sum}
for functions of finite order
\cite{Pfl46,Wei69,Dra81,Ere89a,Ere89b,ES91}; for further discussion, see
\cite{Ere93,Ere02}. The final result was proved by Drasin \cite{Dra87}.

\medskip\noindent
\textbf{Drasin's Theorem.} \emph{Let $f$ be a transcendental meromorphic
function of finite order $\rho$ satisfying \eqref{sum}. Then:
\begin{itemize}
\item[(i)] $2\rho$ is an integer not less than $2$;
\item[(ii)] $\delta(a,f)=p(a)/\rho$, where the $p(a)$ are nonnegative integers.
\end{itemize}}

Drasin's theorem also has a third conclusion, about asymptotic values,
that is not needed
for our discussion.
The first-named author, in collaboration with M.~Sodin, developed an
alternative approach to Drasin's theorem in \cite{Ere89a,Ere89b,ES91},
culminating in the
following result from \cite[p.~1196]{Ere93}.

A positive continuous function $\ell:(0,\infty)\to(0,\infty)$ is called
\emph{slowly varying}, in the sense of Karamata, if
$\ell(cr)/\ell(r)\to1$ uniformly for $c\in[1,2]$ as $r\to\infty$.

\begin{thmA}\label{thm:A}
Let $f$ be a transcendental meromorphic function
of finite lower order $\lambda$ satisfying \eqref{smal}. Then:
\begin{itemize}
\item[(a)] $\rho=\lambda=m/2$, where $m\geq2$ is an integer, and
$T(r)\sim r^\rho\ell(r)$ for a slowly varying function $\ell$;
\item[(b)] $\delta(a,f)=p(a)/\lambda$, $a\in\PP^1$,
where the $p(a)$ are nonnegative
integers satisfying
\[
 \sum_{a\in\PP^1}p(a)=2\lambda,
\]
so that \eqref{sum} holds.
\end{itemize}
\end{thmA}

Since conclusion (b) implies the hypothesis of Drasin's theorem,
Theorem~\ref{thm:A} is both a strengthening and a generalization of that theorem.
Its proof is based on potential theory and does not use Drasin's theorem.

For transcendental meromorphic functions of finite lower order,
\eqref{sum} and \eqref{smal} are equivalent; see
\cite[Theorems~B and~C]{Ere98}. The implication
\eqref{smal}$\Rightarrow$\eqref{sum} follows from
Theorem~\ref{thm:A}, and no other proof of this implication for functions
of finite lower order is known. As we explain below, this equivalence is
specific to the scalar case. The asymptotic behavior of these functions
is described in more detail in \cite{Ere93}.

We now pass to holomorphic curves. Throughout the rest of the paper,
we fix an integer $n\geq1$ and denote complex projective space of dimension $n$ by $\PP^n$. 
A \emph{reduced homogeneous representation} of a curve
$f:\C\to\PP^{n}$ is a vector
\[
 \fv=(f_0,\ldots,f_n),\qquad f=[\fv],
\]
of entire functions with no zeros common to all of them.
Two reduced representations of
the same curve differ by multiplication by a zero-free entire function.
A curve is \emph{linearly non-degenerate} if its image is not contained in
a hyperplane, equivalently if its coordinate functions are linearly
independent. A curve is called \emph{rational}
if it extends holomorphically to
$\PP^1$; otherwise it is \emph{transcendental}. Rational curves are exactly
those which allow homogeneous representations with all polynomial coordinates.

The Cartan--Nevanlinna characteristic is
\[
 T(r,f)=T(r,\fv):=
 \frac1{2\pi}\int_{-\pi}^{\pi}\log\norm{\fv(re^{it})}\,\mathrm dt
 -\log\norm{\fv(0)},
\]
where $\norm{\cdot}$ is the Euclidean norm. This definition does not
depend on the reduced representation. The order and lower order of the
curve are defined in terms of this characteristic as in the scalar case.

For a rational curve we have $T(r,f)=O(\log r)$, while for a transcendental
one,
$T(r,f)/\log r\to\infty$, $r\to\infty$.

The analog of the critical-point counting function is the counting
function of inflection points. It is traditionally called the \emph{ramification term}; this terminology
will be used only for the counting function, not for the underlying points. More precisely, let
\[
 W_{\fv}=W(f_0,\ldots,f_n)
 :=\det\bigl(f_j^{(i)}\bigr)_{0\leq i,j\leq n}
\]
be the Wronskian of a reduced representation. If $n_1(r)$ counts its zeros
in $|z|\leq r$, with multiplicity, then
\[
 \begin{split}
 N_1(r,f)&=\int_0^r\bigl(n_1(t)-n_1(0)\bigr)\frac{\mathrm dt}{t}
                 +n_1(0)\log r\\
 &=\frac1{2\pi}\int_{-\pi}^{\pi}\log|W_{\fv}(re^{it})|\,\mathrm dt+c(\fv).
 \end{split}
\]
This counting function is independent of the reduced representation.
Indeed, if $\widetilde{\fv}=g\fv$ is another reduced representation,
where $g$ is an entire function without zeros, then
\[
W_{\widetilde{\fv}}=g^{n+1}W_{\fv},
\]
so $W_{\widetilde{\fv}}$ and $W_{\fv}$ have the same divisor of zeros.
Except in one paragraph below explicitly mentioning degenerate curves, all
curves in this paper are assumed to be linearly non-degenerate; thus
$W_{\fv}\not\equiv0$.

Proximity functions and deficiencies are defined for hyperplanes in
$\PP^n$; see \cite{Car33,Ru21}. Cartan's Second Main Theorem
\cite[Theorem~7.1, equation~(7.3)]{GH04}
implies the following defect relation for transcendental curves of finite order:
\begin{equation}\label{cartan}
 \sum_{a\in A}\delta(a,f)
 +\limsup_{r\to\infty}\frac{N_1(r,f)}{T(r,f)}\le n+1.
\end{equation}
for a transcendental curve of finite order. 
Here $A$ is any set of
hyperplanes in general position: the intersection of any \(n+1\) of them is
empty. Such a system of hyperplanes is also called \emph{admissible}. 
Without the
finite-order assumption, the Second Main Theorem underlying \eqref{cartan}
requires
an exceptional set of finite length.

The following conjecture was stated in \cite[Section~1]{Ere98}; see also \cite{Ere15}.

\begin{conj}\label{conj:regularity}
If a linearly non-degenerate holomorphic curve of finite lower order
satisfies \eqref{smal}, then its order and lower order coincide and their
common value is rational.
\end{conj}

Conclusion (b) of Theorem~\ref{thm:A} does not extend to
target dimensions at least
two. Examples in \cite{Ere98}
have finite order and $N_1\equiv0$, but no
admissible system $A$ satisfies
\[
 \sum_{a\in A}\delta(a,f)=n+1.
\]

If the linear non-degeneracy condition is dropped, the defect relation becomes
\[
 \sum_{a\in A}\delta(a,f)\leq2n,
\]
provided the curve is nonconstant and its image is not contained in the
union of the hyperplanes of the admissible system $A$.
This was proved for the first time by E.~Nochka \cite{Noc83}; complete proofs
are given in
\cite{Voj97,Ru21}, and an alternative proof is contained in \cite{ES92}.
In this setting, \cite{Ere98} describes the curves with maximal
deficiency sum $2n$, in a form analogous to conclusions (a) and (b)
of Theorem~\ref{thm:A}.

We concentrate in this paper on the regularity conclusion (a).
Our main result
settles Conjecture~\ref{conj:regularity}. 

\begin{thm}\label{thm:main}
Let $f:\C\to\PP^{n}$ be a transcendental linearly non-degenerate
holomorphic curve of finite lower order satisfying
\begin{equation}\label{eq:mainhyp}
 N_1(r,f)=o\bigl(T(r,f)\bigr).
\end{equation}
Then its order and lower order coincide, and their common value $\rho$
belongs to the set
\begin{equation}\label{ord}
 \mathcal R_{n+1}:=\left\{1+\frac{k}{q}:\ k\in\Z_{\ge0},\
                           q\in\Z,\ 2\leq q\leq n+1\right\}.
\end{equation}
Moreover,
\begin{equation}\label{reg}
 T(r,f)=r^\rho\ell(r),
\end{equation}
where $\ell$ is slowly varying. 
\end{thm}

The set of orders in \eqref{ord} is sharp: for every
\(\rho\in\mathcal R_{n+1}\) there are transcendental linearly
non-degenerate curves with \(W_{\fv}\equiv1\) and
\(T(r,\fv)\sim c r^\rho\) for some \(c>0\).
Such curves are obtained by taking a fundamental system of solutions of
\(w^{(n+1)}=z^k w^{(n+1-q)}\) as homogeneous coordinates.
A direct proof of the required growth estimates is given in
Proposition~\ref{prop:sharpness-orders}.

We also obtain the following quantitative result for curves of order zero.

\begin{cor}\label{cor:zero-ratio}
If $f:\C\to\PP^{n}$ is a transcendental linearly non-degenerate curve
of order zero, then
\[
 \limsup_{r\to\infty}\frac{N_1(r,f)}{T(r,f)}\geq1.
\]
\end{cor}

This estimate is best possible. An explicit example of the form
$\fv=(1,z,\ldots,z^{n-1},g)$ is given in Proposition~\ref{prop:sharpness-zero}.

The proof of the first conclusion of Theorem~\ref{thm:main},
that the order and lower order coincide and are rational,
is completed in Section~\ref{sec:indices}. But the case of zero order
requires a separate argument. A similar difficulty arises in \cite{Ere93},
where the result of D. Shea \cite{She85} was used to treat the zero order case.

Section~\ref{sec:smallorder} treats the zero-order case.
There, Proposition~\ref{prop:zero-order-ramification} proves
Corollary~\ref{cor:zero-ratio}.

The argument of F.~Nevanlinna described above has been generalized to
holomorphic curves by V.~P.~Petrenko \cite{Pet84}. We recall it because it
motivates our proof of Theorem~\ref{thm:main}. Suppose that
$\fv=(f_0,\ldots,f_n)$ is a reduced representation
and $W_{\fv}$ is free of zeros. By expanding the determinant in
\[
 \begin{vmatrix}
 w&f_0&\cdots&f_n\\
 w'&f_0'&\cdots&f_n'\\
 \vdots&\vdots&&\vdots\\
 w^{(n+1)}&f_0^{(n+1)}&\cdots&f_n^{(n+1)}
 \end{vmatrix}=0
\]
along the first column, the coefficient of $w^{(n+1)}$ is
$(-1)^{{n+1}+2}W_{\fv}$. Division by this coefficient gives a monic differential
equation
\begin{equation}\label{oper}
 w^{(n+1)}+a_1w^{(n)}+\cdots+a_nw'+a_{n+1}w=0
\end{equation}
with entire coefficients and fundamental system $f_0,\ldots,f_n$.
For a curve of finite order, one can choose a reduced representation all of
whose coordinates have finite order. Multiplication by
$W_{\fv}^{-1/(n+1)}$ gives another such representation and eliminates the
coefficient of $w^{(n)}$. M.~Frei's theorem \cite{Fre53} then implies that the
remaining coefficients are polynomials. Standard asymptotic integration (see, for example, \cite{Was65})
then gives coincident order and lower order, with a value in
$\mathcal R_{n+1}$ in the transcendental case, and regular growth of $T(r)$.

Our proof extends this differential equation approach to the case
when the differential equation has infinitely many singularities. 
Three features are important:
\begin{enumerate}[label=(\arabic*)]
\item
The poles of the coefficients of the differential equation
must be controlled collectively,
with estimates that do not deteriorate when poles coalesce.

\item
We rescale at P\'olya peaks and study the resulting distributional limits
of logarithms of moduli, following the potential-theoretic approach
developed in
\cite{Ere89a,Ere89b,ES91,ES92,Ere93}. A new ingredient is the
convergence in measure of suitably rescaled logarithmic derivatives,
proved in Lemma~\ref{lem:logderivlimit}.

\item
The algebraic part of the argument relies on universal Pl\"ucker
coordinates recently developed by S.~N.~Karp and K.~Purbhoo \cite{KP26}. We use their results,
together with Purbhoo's fundamental operator formalism
\cite[Proposition~2.2]{Pur23}. These results are closely related
to the Bethe algebra correspondence developed by
E.~Mukhin, V.~Tarasov and A.~Varchenko \cite{MTV13}. They are used in
Proposition~\ref{prop:localcompact} to obtain the uniform estimates needed
near the poles of the coefficients.
\end{enumerate}

\medskip
\noindent\textbf{Organization of the paper.}
Section~\ref{sec:analytic} collects the analytic preliminaries.
Section~\ref{sec:algebra} establishes two estimates for polynomial solution
spaces; Section~\ref{sec:compactness} transfers them to local holomorphic
representations. Section~\ref{sec:indices} proves that
$\rho=\lambda\in\mathcal R_{n+1}\cup\{0\}$, and
Section~\ref{sec:smallorder} excludes order zero under the hypotheses of the
main theorem. Section~\ref{sec:regular} proves regular variation and completes
the proof of Theorem~\ref{thm:main}. Section~\ref{sec:sharpness} gives the
sharpness examples.

\medskip
\noindent\textbf{Acknowledgments and AI tools disclosure.}
We thank Mikhail Sodin for helpful comments.

Teng Zhang is supported by the China Scholarship Council, the Young Elite
Scientists Sponsorship Program for PhD Students (China Association for Science
and Technology), and the Fundamental Research Funds for the Central
Universities at Xi'an Jiaotong University (Grant No.~xzy022024045).

ChatGPT was used for English-language editing, proofreading, and grammatical
correction, and as an exploratory tool for discussing possible approaches to
selected results under the authors' mathematical supervision and guidance.
The authors take full responsibility for all mathematical arguments and for
the accuracy and correctness of the final manuscript.

\section{Analytic preliminaries}\label{sec:analytic}

Put $\kappa=n(n+1)/2$, the sum of the derivative orders in an
$(n+1)\times(n+1)$ Wronskian. We write
$D(a,R)=\{z:|z-a|<R\}$, $D_R=D(0,R)$, and $\mathrm dA$ for planar area
measure. We write $\delta_{ij}$ for the Kronecker delta, equal to $1$
when $i=j$ and to $0$ otherwise. For a holomorphic function $h$ on a neighborhood of $\overline D_r$,
let $M(r,h)=\max_{|z|\le r}|h(z)|$. For a function integrable on $|z|=r$, set
$\mean r u=(2\pi)^{-1}\int_{-\pi}^{\pi}u(re^{i\theta})\,\mathrm d\theta$.
We use $\log^+x=\max\{\log x,0\}$ and
$\log^-x=\max\{-\log x,0\}$ for $x>0$, with
$\log^+0=0$ and $\log^-0=+\infty$.

Convergence in measure is always local and refers to planar area measure.
Values at isolated poles are immaterial. We shall repeatedly use the fact
that a product tends to zero in measure when one factor does so and the
other factors are bounded in measure. All products of limiting derivatives
below are products of measurable functions, not products of distributions.

\subsection{Counting functions and homogeneous coordinates}

For a nonzero entire function $H$, let $n_H(r)$ count its zeros in
$\overline D_r$, with multiplicity, and put
\begin{equation}\label{eq:zero-count-definition}
 N_H(r)=\int_0^r\frac{n_H(t)-n_H(0)}t\,\mathrm dt+n_H(0)\log r.
\end{equation}
In particular, $N_1(r,f)=N_{W_{\fv}}(r)$, with the convention fixed in the
Introduction. If $H(z)=cz^{m_0}+O(z^{m_0+1})$, $c\ne0$, Jensen's formula reads
\begin{equation}\label{eq:zerojensen}
 \mean r{\log|H|}=N_H(r)+\log|c|.
\end{equation}
The definition \eqref{eq:zero-count-definition} also gives
\begin{equation}\label{eq:countbound}
 n_H(r)\log2\le N_H(2r),\qquad r\ge1.
\end{equation}

A fixed unitary change of coordinates preserves $T(r,f)$ and $|W_{\fv}|$.
We make such a change so that $\fv(0)/\norm{\fv(0)}=(n+1)^{-1/2}(1,\ldots,1)$.
In particular, every $f_j(0)$ is nonzero. Write
$\TN(r,h)=m(r,h)+N(r,h)$ for the usual scalar Nevanlinna characteristic,
where $m(r,h)=\mean r{\log^+|h|}$ and $N(r,h)$ counts poles. For coordinate
quotients, Jensen's formula \eqref{eq:zerojensen} and
$N(r,f_\ell/f_j)\le N_{f_j}(r)$ give
\begin{equation}\label{eq:quotientbound}
 \TN\!\left(r,\frac{f_\ell}{f_j}\right)
 \le T(r,f)+\log\norm{\fv(0)}-\log|f_j(0)|.
\end{equation}
Here we used the pointwise inequality
$\log^+|f_\ell/f_j|\le\log\norm{\fv}-\log|f_j|$.

\subsection{The fundamental operator and its normalization}

Let $D=\mathrm d/\mathrm dz$. We first recall the local construction of the
fundamental differential operator, the local counterpart of \eqref{oper}.
For the corresponding construction for spaces of rational functions, see
\cite[Proposition~2.2]{Pur23}.

\begin{lem}\label{lem:fundamental-operator}
Let $\gv=(g_0,\ldots,g_n)$ be holomorphic on a domain $\Omega\subset\C$,
with $W(\gv)\not\equiv0$. There is a unique monic meromorphic differential
operator
\begin{equation}\label{eq:fundamental-operator}
 L_{\gv}=D^{n+1}+\sum_{q=1}^{n+1} b_qD^{n+1-q}
\end{equation}
annihilating every $g_j$. Its coefficients have poles only at zeros of
$W(\gv)$, and $b_1=-W(\gv)'/W(\gv)$.
\end{lem}
\begin{proof}
Away from the Wronskian zeros, the equations
$g_j^{(n+1)}+\sum_{q=1}^{n+1} b_qg_j^{(n+1-q)}=0$ form an invertible linear system
for the $b_q$. Cramer's rule proves existence, uniqueness, and the asserted
meromorphic continuation. When the Wronskian is differentiated, all terms
except the one obtained by differentiating its last row vanish. Substitution
of the differential equation in that row gives
$W(\gv)'=-b_1W(\gv)$.
\end{proof}

The coefficient of $D^{n}$ in \eqref{eq:fundamental-operator} can be eliminated by a scalar gauge
transformation. The resulting operator depends on the curve rather than on
its reduced representation.

\begin{lem}\label{lem:canonical-gauge}
Let $\Omega$, $\gv$, and $L_{\gv}$ be as in
Lemma~\ref{lem:fundamental-operator}. Let $\Omega_0$ be a simply connected
subdomain of $\Omega$ containing no zero of $W(\gv)$. Choose a holomorphic
branch $\eta=W(\gv)^{-1/(n+1)}$ on $\Omega_0$, and put
$\widehat{\gv}=\eta\gv$. Then
$W(\widehat{\gv})=1$, and its fundamental operator is
\begin{equation}\label{eq:canonical}
 \widehat L_{\gv}=D^{n+1}+\sum_{q=2}^{n+1} Q_qD^{n+1-q}.
\end{equation}
The coefficients $Q_q$ are single-valued meromorphic functions on $\Omega$,
with poles only at zeros of $W(\gv)$. They are unchanged by a constant
invertible change of basis or by multiplication of $\gv$ by a zero-free
holomorphic function. Under $z\mapsto tz$, $t>0$, their transformation law is
\begin{equation}\label{eq:scalecoeff}
 Q_q(z)\longmapsto t^qQ_q(tz),\qquad 2\le q\le n+1.
\end{equation}
\end{lem}
\begin{proof}
Leibniz's rule expresses the Wronskian matrix of $\eta\gv$ as a lower
triangular matrix, with diagonal entries $\eta$, times the Wronskian matrix
of $\gv$. Thus
\begin{equation}\label{eq:wronskian-gauge}
 W(\eta\gv)=\eta^{n+1}W(\gv).
\end{equation}
Lemma~\ref{lem:fundamental-operator} now gives \eqref{eq:canonical}.
More explicitly, $\widehat L_{\gv}=\eta L_{\gv}\eta^{-1}$; its coefficients
are differential polynomials in the coefficients $b_q$ of $L_{\gv}$
and the logarithmic derivative $W(\gv)'/W(\gv)$. This proves
meromorphic continuation across the Wronskian zeros. Different branches of
$\eta$ differ by a constant $(n+1)$st root of unity and give the same operator.

Formula \eqref{eq:wronskian-gauge} shows that a zero-free scalar change of
representation leaves the normalized solution space unchanged. A constant
basis change multiplies the Wronskian by its determinant and changes the
normalized basis by a constant invertible matrix. Finally,
$D^k(v(tz))=t^kv^{(k)}(tz)$ gives \eqref{eq:scalecoeff}.
\end{proof}

We henceforth reserve $Q_2,\ldots,Q_{n+1}$ for the coefficients of the
original curve $f$ supplied by Lemma~\ref{lem:canonical-gauge}.

\subsection{Logarithmic derivatives at a fixed radius}

The fixed-radius estimate is classical: the first-derivative case is due to
Nevanlinna, and the higher-derivative extension to Hiong \cite{Hio55}.
We quote its precise formulation from Li \cite[Theorem~A]{Li11}.
Unlike a bound with an unspecified constant depending on the function, this
form can be applied uniformly to a sequence of functions.

\begin{lem}[{\cite[Theorem~A]{Li11}}]\label{lem:NH}
For each integer $k\ge1$ there is a constant $C_k$ with the following
property. If $h$ is meromorphic in $D_R$, $0<|h(0)|<\infty$, and
$0<r<\varrho<R\le\infty$, then
\[
 m\!\left(r,\frac{h^{(k)}}h\right)
 \le C_k\left(1+\log^+\TN(\varrho,h)
 +\log^+\log^+\frac1{|h(0)|}+\log^+\varrho
 +\log^+\frac1{\varrho-r}+\log^+\frac1r\right).
\]
\end{lem}

Taking $r=1$ and $\varrho=4$ in Lemma~\ref{lem:NH} gives
\begin{equation}\label{eq:logderivbound}
 m\!\left(1,\frac{h^{(k)}}h\right)
 \le C_k\log\bigl(2+\TN(4,h)+|\log|h(0)||\bigr)
\end{equation}
for $h$ meromorphic near $\overline D_4$ and nonzero and finite at the origin.
The change in $C_k$ absorbs only numerical constants.

\subsection{Limits of normalized logarithmic derivatives}

We need a local convergence statement in addition to
\eqref{eq:logderivbound}. Its first-derivative part follows from convergence
of logarithmic potentials and their Cauchy transforms; related arguments
appear in Bergkvist--Rullg{\aa}rd \cite[Section~4]{BR02}. We include the proof
because the higher-derivative statement and its normalization are used
repeatedly.

Set $\partial=\tfrac12(\partial_x-i\partial_y)$ and
$\Delta=\partial_x^2+\partial_y^2$. We write $W^{1,p}_{\loc}$ for the local
Sobolev space of functions whose first distributional derivatives belong
locally to $L^p$.

\begin{lem}\label{lem:logderivlimit}
Let $h_\nu\not\equiv0$ be holomorphic on a domain $\Omega$, and let
$s_\nu\to\infty$. Suppose that
\begin{equation}\label{eq:log-limit-assumption}
 u_\nu=s_\nu^{-1}\log|h_\nu|\longrightarrow u
 \quad\hbox{in }L^1_{\loc}(\Omega),
\end{equation}
where $u$ is subharmonic and not identically $-\infty$. Then
$u\in W^{1,p}_{\loc}(\Omega)$ for $1\le p<2$, and for every integer $k\ge1$,
\begin{equation}\label{eq:logderivlimit}
 \frac{h_\nu^{(k)}}{s_\nu^kh_\nu}
 \convmeas(2\partial u)^k
 \quad\hbox{locally in }\Omega.
\end{equation}
For $k=1$, the convergence also holds in $L^p_{\loc}$ for $1\le p<2$.
\end{lem}
\begin{proof}
\step{Step 1. The first derivative.}
The Riesz measures $\mu_\nu=(2\pi)^{-1}\Delta u_\nu$ are the zero-counting
measures of $h_\nu$, divided by $s_\nu$. Assumption
\eqref{eq:log-limit-assumption} implies their weak convergence on compact
subsets to $\mu=(2\pi)^{-1}\Delta u$. Their masses are consequently bounded
on every compact subset of $\Omega$.

Fix $K\Subset\Omega$, and choose a smooth compactly supported function
$\chi\in C_c^\infty(\Omega)$ with $0\le\chi\le1$ and $\chi=1$ near $K$.
Define $v_\nu(z)=\int\chi(a)\log|z-a|\,\mathrm d\mu_\nu(a)$ and define $v$
by replacing $\mu_\nu$ with $\mu$. The distributions
$H_\nu=u_\nu-v_\nu$ and $H=u-v$ have harmonic representatives on a fixed
neighborhood of $K$.

Truncate the kernel $(z-a)^{-1}$ smoothly within distance $2\delta$ of the
diagonal. The truncated integrals converge uniformly on $K$, by weak
convergence and uniform continuity of the truncated kernel. The removed
part satisfies, for $1\le p<2$,
\begin{equation}\label{eq:kernel-truncation-bound}
 \left\|\int_{|z-a|<2\delta}\frac{\chi(a)}{z-a}
                  \,\mathrm d\mu_\nu(a)\right\|_{L^p(K)}
 \le C_{K,p}\delta^{(2-p)/p}.
\end{equation}
This follows from Minkowski's inequality and the uniform mass bound; the
same estimate holds for $\mu$. Letting $\delta\downarrow0$ in
\eqref{eq:kernel-truncation-bound} proves
$\partial v_\nu\to\partial v$ in $L^p(K)$. Truncation of the logarithmic
kernel likewise gives $v_\nu\to v$ in $L^1_{\loc}$. Hence
$H_\nu\to H$ in $L^1$ near $K$, and interior harmonic estimates give
convergence of all derivatives on smaller neighborhoods. Since
$2\partial u_\nu=h_\nu'/(s_\nu h_\nu)$ almost everywhere, we obtain
\begin{equation}\label{eq:first-logderiv-limit}
 \frac{h_\nu'}{s_\nu h_\nu}\longrightarrow2\partial u
 \quad\hbox{in }L^p_{\loc}(\Omega),\qquad 1\le p<2.
\end{equation}
The logarithmic potential itself belongs locally to every finite $L^p$.
Together with \eqref{eq:first-logderiv-limit}, this proves the Sobolev
assertion.

\step{Step 2. Derivatives of the first logarithmic derivative.}
Put $L_\nu=h_\nu'/h_\nu$. Differentiating
$u_\nu=v_\nu+H_\nu$, with $v_\nu$ and $H_\nu$ constructed in Step~1,
away from the zeros gives, near $K$,
\begin{equation}\label{eq:singular-logderivative}
 L_\nu^{(j)}(z)=(-1)^jj!\sum_a
 \frac{\chi(a)\ord_a(h_\nu)}{(z-a)^{j+1}}+R_{\nu,j}(z),
 \qquad \sup_K|R_{\nu,j}|=O(s_\nu).
\end{equation}
Here $R_{\nu,j}=2s_\nu\partial^{j+1}H_\nu$. The sum is over the
distinct zeros in $\supp\chi$, with multiplicities included through
$\ord_a(h_\nu)$; their total multiplicity is $O(s_\nu)$.
For $j\ge1$, choose $1/(j+1)<\alpha<\min\{1,2/(j+1)\}$. Subadditivity of
$x^\alpha$ and local integrability of $|z-a|^{-(j+1)\alpha}$ give, from
\eqref{eq:singular-logderivative},
\[
 \int_K\left|\frac{L_\nu^{(j)}}{s_\nu^{j+1}}\right|^\alpha\dA
 \le C\bigl(s_\nu^{1-(j+1)\alpha}+s_\nu^{-j\alpha}\bigr)\longrightarrow0.
\]
Thus
\begin{equation}\label{eq:higher-logderiv-measure}
 L_\nu^{(j)}/s_\nu^{j+1}\convmeas0,\qquad j\ge1.
\end{equation}

\step{Step 3. Higher derivatives of $h_\nu$.}
For a nonzero meromorphic function $h$, put $L=h'/h$ and $P_k=h^{(k)}/h$,
with $P_0=1$. The identity $P_{k+1}=P_k'+LP_k$ shows inductively that $P_k$
is a differential polynomial of weight $k$ when $L^{(j)}$ is assigned weight
$j+1$. Its term without derivatives of $L$ is $L^k$, with coefficient one.
Every other monomial contains at least one $L^{(j)}$ with $j\ge1$.
After division by $s_\nu^k$, each monomial is a product of factors
$L_\nu^{(j)}/s_\nu^{j+1}$. Equations \eqref{eq:first-logderiv-limit} and
\eqref{eq:higher-logderiv-measure} show that the leading monomial tends to
$(2\partial u)^k$ and every other monomial tends to zero in measure. This
proves \eqref{eq:logderivlimit}.
\end{proof}

The powers in \eqref{eq:logderivlimit} need not be locally integrable when
$k\ge2$. For example, $u(z)=\log|z|$ gives $2\partial u=1/z$ almost
everywhere, while $1/z^2\notin L^1_{\loc}$. Convergence in measure is the
appropriate conclusion for our limiting differential equations.

\subsection{Wronskians and subharmonic compactness}

The following consequence will provide a lower bound for the limiting
component logarithms.

\begin{lem}\label{lem:sum}
Let $\hv_\nu=(h_{\nu,0},\ldots,h_{\nu,n})$ be holomorphic on a domain
$\Omega$, and let $s_\nu\to\infty$. Suppose that, for every
$j\in\{0,\ldots,n\}$,
$s_\nu^{-1}\log|h_{\nu,j}|\to v_j$ in $L^1_{\loc}(\Omega)$, where the $v_j$ are
subharmonic and not identically $-\infty$, and that
$s_\nu^{-1}\log|W(\hv_\nu)|\convmeas0$. Then
\begin{equation}\label{eq:sum-positive}
 \sum_{j=0}^{n}v_j\ge0\quad\hbox{almost everywhere in }\Omega.
\end{equation}
\end{lem}
\begin{proof}
Factoring each component from its Wronskian column gives
\[
 D_\nu:=\frac{W(\hv_\nu)}{s_\nu^\kappa\prod_j h_{\nu,j}}
 =\det\left(\frac{h_{\nu,j}^{(i)}}{s_\nu^ih_{\nu,j}}\right)_{0\le i,j\le n}
 \quad\hbox{almost everywhere}.
\]
Lemma~\ref{lem:logderivlimit} makes every entry, and hence $D_\nu$, bounded
in measure. Therefore $s_\nu^{-1}\log^+|D_\nu|\convmeas0$. Taking logarithms
in this identity yields
\[
 \sum_j s_\nu^{-1}\log|h_{\nu,j}|
 \ge s_\nu^{-1}\log|W(\hv_\nu)|
       -\kappa s_\nu^{-1}\log s_\nu-s_\nu^{-1}\log^+|D_\nu|.
\]
Pass to a subsequence on which all convergences hold almost everywhere.
The resulting inequality is \eqref{eq:sum-positive}.
\end{proof}

If $U=\max_jv_j$, then \eqref{eq:sum-positive} implies
\begin{equation}\label{eq:sandwich}
 U\ge0,\qquad -nU\le v_j\le U\quad\hbox{almost everywhere}.
\end{equation}
We also use the following standard compactness theorem and its upper-bound
consequence.

\begin{lem}[{\cite[Theorem~4.1.9(a),(b)]{Hor03}}]
\label{lem:subharmonic-compactness}
Let $(u_\nu)$ be a sequence of subharmonic functions on a domain
$\Omega\subset\C$, locally uniformly bounded above. It has a subsequence
that either tends locally uniformly to
$-\infty$ or converges in $L^1_{\loc}(\Omega)$ to a subharmonic function not
identically $-\infty$. If $u_\nu\to u$ in $L^1_{\loc}(\Omega)$ and
$K\Subset V\Subset\Omega$, then
\[
 \limsup_{\nu\to\infty}\sup_Ku_\nu\le\sup_Vu.
\]
\end{lem}

For a vector $\hv=(h_0,\ldots,h_n)\in\C^{n+1}\setminus\{0\}$,
we pass between the norm and the component logarithms using
\begin{equation}\label{eq:norm-max}
 0\le\log\norm{\hv}-\max_j\log|h_j|\le\tfrac12\log(n+1).
\end{equation}
In particular, finite $L^1_{\loc}$ limits of all normalized component
logarithms determine the corresponding norm limit.

\section{Estimates for polynomial solution spaces}\label{sec:algebra}

We derive two estimates from the universal Pl\"ucker-coordinate formula of
Karp--Purbhoo \cite{KP26}. Both are uniform when zeros of the Wronskian
coalesce. One controls the differential coefficients; the other controls a
solution in terms of its initial derivatives. The representation-theoretic results are stated in
Lemmas~\ref{lem:KP-correspondence} and~\ref{lem:character-projection};
Lemma~\ref{lem:universal-minors} gives their analytic consequence.

\subsection{Wronskian minors}

Let $V\subset\C[z]$ have dimension $n+1$, and choose a basis
$v_0,\ldots,v_n$. For a partition
$\tau=(\tau_1\ge\cdots\ge\tau_{n+1}\ge0)$, put
$|\tau|=\sum_i\tau_i$ and define
\begin{equation}\label{eq:derivative-minor}
 \Delta^\tau(V;a)=\det\bigl(v_j^{(i_k)}(a)\bigr)_{
                    1\le k\le n+1,\ 0\le j\le n},
 \qquad i_k=k-1+\tau_{n+2-k}.
\end{equation}
Trailing zero parts are omitted when convenient, and $\varnothing$ denotes
the empty partition. Thus $\Delta^\varnothing(V;a)$ is the Wronskian. A
basis change multiplies all these minors by the same nonzero constant, so
their ratios are well defined independently of the basis.

For the translated space $V_a=\{p(a+z):p\in V\}$, the coefficient of
$z^r/r!$ in a basis element is its $r$th derivative at $a$. Accordingly,
\eqref{eq:derivative-minor} is precisely the Pl\"ucker minor of $V_a$ in the
factorial basis $1,z,z^2/2!,\ldots$. The column indices in the convention of
Karp--Purbhoo \cite[equation~(2.5)]{KP26} are
$k+\tau_{n+2-k}$, $1\le k\le n+1$. There is no additional factorial or sign
in passing to our minors.

Let $P(z)=\prod_{\ell=1}^M(z-a_\ell)$ be the monic normalization of the
Wronskian, with roots repeated according to multiplicity. Write
$e_s(x_1,\ldots,x_M)=\sum_{|I|=s}\prod_{\ell\in I}x_\ell$ for the
$s$th elementary symmetric polynomial; the sum is over subsets of
$\{1,\ldots,M\}$. We set $e_0=1$ and $e_s=0$ for $s>M$.

\subsection{The Pl\"ucker-coordinate identities}

We recall the translation formula and the spectral description of the
Wronski map. For partitions $\tau$ and
$\lambda$, write $\tau\subseteq\lambda$ if $\tau_i\leq\lambda_i$ for every
$i$, adding zero parts when necessary. The Young diagram of $\lambda$ has
$\lambda_i$ boxes in row $i$. If $\tau\subseteq\lambda$, let
$f^{\lambda/\tau}$ be the number of fillings of the boxes of the skew diagram
$\lambda/\tau$ by $1,\ldots,|\lambda|-|\tau|$ that increase along rows and
columns. Put $f^\lambda=f^{\lambda/\varnothing}$ and
$f^{\lambda/\lambda}=1$.

Fix an integer $m\geq n+1$, and let $\C_{m-1}[z]$ be the polynomials of
degree at most $m-1$. For a partition
$\omega=(\omega_1,\ldots,\omega_{n+1})$ with
$\omega_1\leq m-n-1$, let $\mathcal X^\omega$ denote the Schubert cell
consisting of the $(n+1)$-dimensional subspaces of $\C_{m-1}[z]$ that have a
basis of respective degrees $j+\omega_{n+1-j}$, $0\leq j\leq n$.
If $V\in\mathcal X^\omega$, then $\Delta^\omega(V;a)$ is a nonzero
constant in $a$, and the Wronskian degree is $M=|\omega|$. Define its
normalized Pl\"ucker coordinates by
\begin{equation}\label{eq:normalized-Plucker-coordinates}
 \widetilde\Delta^\tau(V;a)
 =\frac{M!}{f^\omega}
   \frac{\Delta^\tau(V;a)}{\Delta^\omega(V;a)}.
\end{equation}
Coordinates indexed by partitions not contained in $\omega$ are zero.
For a partition with more than $n+1$ nonzero parts, the coordinate is
understood to be zero. Formula~\eqref{eq:normalized-Plucker-coordinates}
fixes the common scalar by requiring
$\widetilde\Delta^\omega(V;a)=M!/f^\omega$.

The first result describes translation in these coordinates.

\begin{lem}[{\cite[Proposition~2.9]{KP26}}]\label{lem:plucker-translation}
Let $V\in\mathcal X^\omega$, with the normalized coordinates in
\eqref{eq:normalized-Plucker-coordinates}. For every partition $\tau$ and
all $a,t\in\C$,
\begin{equation}\label{eq:Plucker-translation}
 \widetilde\Delta^\tau(V;a+t)
 =\sum_{\lambda\supseteq\tau}
   \frac{f^{\lambda/\tau}}{(|\lambda|-|\tau|)!}
   t^{|\lambda|-|\tau|}\widetilde\Delta^\lambda(V;a).
\end{equation}
Only partitions $\lambda\subseteq\omega$ contribute to the sum.
\end{lem}

We next define the group-algebra elements used in the second result.
Let $\mathfrak S_M$ be the symmetric group on $\{1,\ldots,M\}$.
For $I\subseteq\{1,\ldots,M\}$, let $\mathfrak S_I$ be its subgroup
permuting $I$ and fixing its complement. If $|I|=s$ and $|\tau|=s$, let
$\chi^\tau$ be the irreducible character of $\mathfrak S_s$ indexed by
$\tau$, transported to $\mathfrak S_I$ by the increasing ordering of $I$.
The corresponding irreducible representation, or Specht module, has
dimension $f^\tau$. For arbitrary complex parameters $z_1,\ldots,z_M$, put
\begin{equation}\label{eq:KP-operators}
 \alpha_I^\tau=\sum_{\sigma\in\mathfrak S_I}\chi^\tau(\sigma)\sigma,
 \qquad
 \beta^\tau(a)=
 \sum_{\substack{I\subseteq\{1,\ldots,M\}\\ |I|=|\tau|}}
 \alpha_I^\tau\prod_{\ell\notin I}(a+z_\ell)
 \quad\text{in }\C[\mathfrak S_M].
\end{equation}
Thus $\beta^\tau=0$ when $|\tau|>M$, whereas
$\beta^\varnothing(a)=\prod_{\ell=1}^M(a+z_\ell)\,1$.
Here $\C[\mathfrak S_M]$ is the complex group algebra and $1$ its identity.
We use the definition \eqref{eq:KP-operators} from
Karp--Purbhoo \cite[equation~(1.2)]{KP26}. Let
$\mathcal B_M(z_1,\ldots,z_M)$ be the unital subalgebra generated by the
$\beta^\tau(0)$. A common eigenspace of an algebra is a nonzero subspace
on which each element of the algebra acts as a scalar.

The following statements identify this algebra and its eigenvalues.

\begin{lem}[{\cite[Theorem~1.3(i),(ii),(iv),(v) and Corollary~4.15]{KP26}}]
\label{lem:KP-correspondence}
Let $M\geq1$ be an integer and let $z_1,\ldots,z_M\in\C$. Let
$\beta^\tau(a)$ be defined by \eqref{eq:KP-operators}, and let
$\mathcal B_M(z_1,\ldots,z_M)$ be the unital subalgebra of
$\C[\mathfrak S_M]$ generated by all the elements $\beta^\tau(0)$.
\begin{enumerate}
\item The elements $\beta^\tau(a)$ commute for all $a$ and $\tau$.
They belong to $\mathcal B_M(z_1,\ldots,z_M)$, which is the Bethe algebra,
and for every partition $\tau$ and all $a,t\in\C$,
\begin{equation}\label{eq:KP-translation}
 \beta^\tau(a+t)=\sum_{\lambda\supseteq\tau}
 \frac{f^{\lambda/\tau}}{(|\lambda|-|\tau|)!}
 t^{|\lambda|-|\tau|}\beta^\lambda(a).
\end{equation}
The sum is finite because $\beta^\lambda=0$ for $|\lambda|>M$.
\item Let $\omega$ be a partition of $M$ with at most $n+1$ parts, and
choose $m\geq n+1+\omega_1$. Let $\mathsf M^\omega$ be its Specht module,
with the natural action of $\C[\mathfrak S_M]$. For every common eigenspace
$E\subseteq\mathsf M^\omega$ of $\mathcal B_M(z_1,\ldots,z_M)$, the
eigenvalues of $\beta^\tau(0)$ are the normalized coordinates
$\widetilde\Delta^\tau(V;0)$ of a space $V\in\mathcal X^\omega$ whose
monic Wronskian is $\prod_{\ell=1}^M(z+z_\ell)$.
Every space in $\mathcal X^\omega$ with this monic Wronskian is obtained
from some such eigenspace.
\end{enumerate}
\end{lem}

The norm estimate uses the following character projection formula. A
$\tau$-isotypic subspace means the sum of all invariant subspaces isomorphic
to the irreducible representation indexed by $\tau$.

\begin{lem}[{\cite[Proposition~2.16]{KP26}}]\label{lem:character-projection}
Let $I$ be a finite set of size $s$, let $\mathfrak S_I$ be its symmetric
group, and let $\pi$ be a unitary representation of $\mathfrak S_I$ on a
finite-dimensional Hermitian space $\mathcal H$. Let $\tau$ be a partition
of $s$ and let $\chi^\tau$ be the corresponding irreducible character. With
$C^\tau=\sum_{\sigma\in\mathfrak S_I}\chi^\tau(\sigma)\pi(\sigma)$ and
$f^\tau=\chi^\tau(1)$, the operator $(f^\tau/s!)C^\tau$ is the orthogonal
projection onto the $\tau$-isotypic subspace of $\mathcal H$.
\end{lem}

\subsection{An estimate for the derivative minors}

Lemmas~\ref{lem:plucker-translation}--\ref{lem:character-projection}
give the following identity and bound for derivative minors.

\begin{lem}\label{lem:universal-minors}
Let $V\subset\C[z]$ have dimension $n+1$, and let
$P(z)=\prod_{\ell=1}^M(z-a_\ell)$ be its monic Wronskian, with roots
repeated according to multiplicity. There are a finite-dimensional
Hermitian space $\mathcal H$, a unitary representation
$\pi\colon\mathfrak S_M\to U(\mathcal H)$, and a unit vector
$\ev\in\mathcal H$, fixed for this $V$, such that for every partition
$\tau$ with at most $n+1$ parts and every $a$ with $P(a)\ne0$,
\begin{equation}\label{eq:universal-minors}
 \frac{\Delta^\tau(V;a)}{\Delta^\varnothing(V;a)}
 =\sum_{\substack{I\subseteq\{1,\ldots,M\}\\|I|=s}}
   \frac{\langle C_I^\tau\ev,\ev\rangle}
        {\prod_{\ell\in I}(a-a_\ell)},
 \qquad s=|\tau|.
\end{equation}
Here $U(\mathcal H)$ denotes the unitary group of $\mathcal H$,
$C_I^\tau=\pi(\alpha_I^\tau)$ with $\alpha_I^\tau$ as in
\eqref{eq:KP-operators}, and the inner product is linear in its first
argument. The coefficients obey
\begin{equation}\label{eq:projectionbound}
 |\langle C_I^\tau\ev,\ev\rangle|\leq s!.
\end{equation}
For $M=0$ we use the trivial group, the empty product is $1$, and a sum
with $s>M$ is $0$.
\end{lem}
\begin{proof}
\step{Step 1. The eigenspace associated with $V$.}
First suppose that $M\geq1$. Choose a basis of $V$ with distinct degrees
$\delta_0<\cdots<\delta_n$, and choose $m>\delta_n$. Then
$V\in\mathcal X^\omega$ for
$\omega=(\delta_n-n,\ldots,\delta_1-1,\delta_0)$, and
$|\omega|=\sum_{j=0}^n\delta_j-\kappa=M$.
Apply Lemma~\ref{lem:KP-correspondence}(ii) with this $\omega$, this $m$,
and $z_\ell=-a_\ell$. Let $E\subseteq\mathsf M^\omega$ be the common
eigenspace supplied by that lemma for $V$. Equip $\mathsf M^\omega$ with
an invariant Hermitian inner product, obtained by averaging any Hermitian
inner product over $\mathfrak S_M$. Write $\mathcal H$ for this Hermitian
space and $\pi$ for its unitary representation, and choose a unit vector
$\ev\in E$.

Lemma~\ref{lem:KP-correspondence}(ii) gives the eigenvalues at the center
$0$. Apply its translation identity \eqref{eq:KP-translation} with initial
center $0$ and increment $a$, and apply \eqref{eq:Plucker-translation}
from Lemma~\ref{lem:plucker-translation} with the same center and increment.
The identical finite sums show that
\begin{equation}\label{eq:KP-eigenvalue-at-center}
 \pi(\beta^\tau(a))\ev
 =\widetilde\Delta^\tau(V;a)\ev
 \qquad(a\in\C).
\end{equation}
The eigenspace, representation, and vector were chosen before $a$ and
$\tau$, and do not depend on them.

\step{Step 2. The ratio of the minors.}
For the chosen parameters, \eqref{eq:KP-operators} gives
$\pi(\beta^\varnothing(a))=P(a)\operatorname{Id}_{\mathcal H}$.
Taking inner products in \eqref{eq:KP-eigenvalue-at-center}, dividing by
$P(a)$, and using \eqref{eq:normalized-Plucker-coordinates}, we obtain
\[
 \frac{\Delta^\tau(V;a)}{\Delta^\varnothing(V;a)}
 =\frac{\langle\pi(\beta^\tau(a))\ev,\ev\rangle}{P(a)}
 =\sum_{|I|=s}
   \frac{\langle C_I^\tau\ev,\ev\rangle}
        {\prod_{\ell\in I}(a-a_\ell)}.
\]
This proves \eqref{eq:universal-minors}, including the case $s>M$.
The arbitrary-parameter hypothesis of Lemma~\ref{lem:KP-correspondence}
already includes repeated roots.

\step{Step 3. The uniform coefficient bound and the constant case.}
Apply Lemma~\ref{lem:character-projection} to the restriction of $\pi$ to
$\mathfrak S_I$ and the partition $\tau$. It gives
$\|C_I^\tau\|\leq s!/f^\tau\leq s!$. Since $\|\ev\|=1$, this proves
\eqref{eq:projectionbound}.

If $M=0$, the degree formula
$0=\sum_{j=0}^n\delta_j-\kappa$ and the inequalities $\delta_j\geq j$
force $\delta_j=j$. Thus $V=\C_n[z]$, the space of all polynomials of
degree at most $n$. Every nonempty minor vanishes, and the assertions
follow with the one-dimensional trivial representation.
\end{proof}

Taking absolute values in \eqref{eq:universal-minors} and using
\eqref{eq:projectionbound} gives the estimate that will be used for initial
values:
\begin{equation}\label{eq:minor-es-bound}
 \left|\frac{\Delta^\tau(V;a)}{\Delta^\varnothing(V;a)}\right|
 \le s!\,e_s\bigl(|a-a_1|^{-1},\ldots,|a-a_M|^{-1}\bigr),
 \qquad s=|\tau|.
\end{equation}

\subsection{Differential coefficients and initial values}

For the coefficients of the fundamental operator, it is important to retain
the product form rather than pass to partial fractions.

\begin{prop}\label{prop:polynomialcoeff}
Let $V\subset\C[z]$ have dimension $n+1$, and let $a_1,\ldots,a_M$ be the
zeros of its Wronskian, with multiplicity. Let
$L_V=D^{n+1}+\sum_{q=1}^{n+1}b_qD^{n+1-q}$ be the fundamental operator
of any basis of $V$, as defined in Lemma~\ref{lem:fundamental-operator}.
Its coefficients have representations
\begin{equation}\label{eq:productformula}
 b_q(z)=(-1)^q\sum_{\substack{I\subset\{1,\ldots,M\}\\|I|=q}}
       \frac{\gamma_{q,I}}{\prod_{\ell\in I}(z-a_\ell)},
 \qquad |\gamma_{q,I}|\le q!,\quad 1\le q\le n+1,
\end{equation}
where the $\gamma_{q,I}$ are independent of $z$.
\end{prop}
\begin{proof}
Cramer's rule replaces the Wronskian row of derivative order $n+1-q$ by the
row of order $n+1$. Reordering the rows increasingly identifies the numerator
with the minor indexed by $(1^q)$, the partition consisting of $q$ ones,
and gives
$b_q=(-1)^q\Delta^{(1^q)}/\Delta^\varnothing$.
Choose the representation $\pi$, the unit vector $\ev$, and the operators
$C_I^\tau$ supplied by Lemma~\ref{lem:universal-minors} for this $V$.
Apply \eqref{eq:universal-minors} and \eqref{eq:projectionbound} with
$\tau=(1^q)$ and center $a=z$. These choices are fixed for $V$, so the
resulting coefficients
$\gamma_{q,I}=\langle C_I^{(1^q)}\ev,\ev\rangle$ are constant in $z$.
The empty-sum convention also covers $M=0$.
\end{proof}

The constants in \eqref{eq:productformula} do not depend on the mutual
distances between the roots. A partial-fraction expansion would not retain
this uniformity. The same minor estimate also bounds a normalized basis.

\begin{prop}\label{prop:initial-basis}
Let $V\subset\C[z]$ have dimension $n+1$, let $a_1,\ldots,a_M$ be
its Wronskian zeros counted with multiplicity, and let
$a\in\C\setminus\{a_1,\ldots,a_M\}$. Let
$\psi_0,\ldots,\psi_n$ be the unique basis of $V$ with
$\psi_j^{(i)}(a)=\delta_{ij}$ for $0\leq i,j\leq n$. Then
\begin{equation}\label{eq:basismajorant}
 |\psi_j(a+z)|\le\frac{|z|^j}{j!}
       \prod_{\ell=1}^M\left(1+\frac{|z|}{|a-a_\ell|}\right),
 \qquad z\in\C,\quad 0\le j\le n.
\end{equation}
\end{prop}
\begin{proof}
For $k\ge n+1$, replace row $j$ of the identity matrix
$(\psi_j^{(i)}(a))$ by the row of derivatives of order $k$. Its determinant
is $\psi_j^{(k)}(a)$. After reordering, the row indices are
$0,\ldots,j-1,j+1,\ldots,n,k$, so the partition in
\eqref{eq:derivative-minor} is $(k-n,1^{n-j})$, of size $k-j$.
Since the Wronskian at $a$ is one, \eqref{eq:minor-es-bound} gives
\[
 \frac{|\psi_j^{(k)}(a)|}{k!}
 \le\frac{(k-j)!}{k!}\,e_{k-j}(x_1,\ldots,x_M)
 \le\frac1{j!}e_{k-j}(x_1,\ldots,x_M),
 \qquad x_\ell=|a-a_\ell|^{-1}.
\]
Taylor's formula and the prescribed derivatives of orders below $n+1$ imply
\[
 |\psi_j(a+z)|
 \le\frac{|z|^j}{j!}\sum_{s\ge0}e_s(x_1,\ldots,x_M)|z|^s
 =\frac{|z|^j}{j!}\prod_{\ell=1}^M(1+|z|x_\ell).
\]
This is \eqref{eq:basismajorant}.
\end{proof}

For $V$, $a$, and the normalized basis in
Proposition~\ref{prop:initial-basis}, each $h\in V$ satisfies
$h=\sum_{i=0}^{n}h^{(i)}(a)\psi_i$. Thus
\eqref{eq:basismajorant} gives
\begin{equation}\label{eq:initial-value-bound}
 |h(a+z)|\le
 \left(\sum_{i=0}^{n}\frac{|h^{(i)}(a)|\,|z|^i}{i!}\right)e^{|z|B_V(a)},
 \qquad B_V(a)=\sum_{\ell=1}^M|a-a_\ell|^{-1}.
\end{equation}
We shall apply \eqref{eq:initial-value-bound} after approximating local
holomorphic solution spaces by polynomial spaces.

\section{Local compactness}\label{sec:compactness}

Taylor approximation transfers the estimates of Section~\ref{sec:algebra}
to holomorphic solution spaces. The essential distinction is between the
$o(s_\nu)$ Wronskian zeros in a fixed disk and the $O(s_\nu)$ zeros outside
it. The first group has a negligible contribution after normalization; the
second gives a bounded holomorphic contribution.

\subsection{Compactness of the differential coefficients}

\begin{prop}\label{prop:localcompact}
Let $\gv_\nu=(g_{\nu,0},\ldots,g_{\nu,n})$ be holomorphic on $D_{32}$,
and let $s_\nu\to\infty$. Suppose that $W(\gv_\nu)=P_\nu$ is a monic
polynomial and that, for a fixed $C$,
\begin{equation}\label{eq:localhyp}
 \sup_{D_{32}}\log\norm{\gv_\nu}\le Cs_\nu,
 \qquad m_\nu:=\deg P_\nu=o(s_\nu).
\end{equation}
If $L_\nu=D^{n+1}+\sum_{q=1}^{n+1} b_{q,\nu}D^{n+1-q}$ is the fundamental operator,
the sequences $b_{q,\nu}/s_\nu^q$ are simultaneously relatively compact
locally in measure on $D_4$. Their limits are holomorphic, with bounds on
$D_1$ depending only on $n+1$ and $C$.

Let $\widehat b_{q,\nu}$ be the coefficients after multiplication of the
solution vector by $P_\nu^{-1/(n+1)}$, defined locally and continued
meromorphically. Then $\widehat b_{1,\nu}=0$ and
\begin{equation}\label{eq:gaugecomparison}
 \frac{\widehat b_{q,\nu}-b_{q,\nu}}{s_\nu^q}\convmeas0
 \quad\hbox{locally on }D_4,\qquad 1\le q\le n+1.
\end{equation}
\end{prop}
\begin{proof}
\step{Step 1. Polynomial approximation.}
Fix $A>0$, to be chosen sufficiently large below.
By \eqref{eq:localhyp} and Cauchy's coefficient estimate on $|z|=24$,
$|[z^m]g_{\nu,j}|\le e^{Cs_\nu}24^{-m}$, where $[z^m]$ denotes a Taylor
coefficient. For the Taylor polynomial $p_{\nu,j}$ of degree $N$, it follows
that
\[
 \max_{0\le k\le n+1}\sup_{D_{16}}
 |p_{\nu,j}^{(k)}-g_{\nu,j}^{(k)}|
 \le C_n e^{Cs_\nu}(N+1)^{n+1}(2/3)^N.
\]
Choose $L$ with $L\log(3/2)>A+C+1$ and set $N_\nu=\lceil Ls_\nu\rceil$.
For all sufficiently large $\nu$,
\begin{equation}\label{eq:taylorerror}
 \max_{0\le j\le n}\max_{0\le k\le n+1}\sup_{D_{16}}
 |p_{\nu,j}^{(k)}-g_{\nu,j}^{(k)}|\le e^{-As_\nu}.
\end{equation}
Put $\pv_\nu=(p_{\nu,0},\ldots,p_{\nu,n})$ and $R_\nu=W(\pv_\nu)$.
Cauchy's derivative estimates and \eqref{eq:taylorerror} bound all entries
of the relevant derivative matrices on $D_{16}$ by $e^{C_1s_\nu}$, where
$C_1$ depends only on $n+1,C$. Expanding the difference of two determinants
therefore gives
\begin{equation}\label{eq:wronskiapprox}
 \deg R_\nu\le (n+1)N_\nu=O(s_\nu),\qquad
 \sup_{D_{12}}|R_\nu-P_\nu|\le e^{-Bs_\nu}.
\end{equation}
Here $B$ tends to infinity with $A$. If $\mathcal C^p_{q,\nu}$ and
$\mathcal C^g_{q,\nu}$ are the signed Cramer numerators for the two
fundamental operators, the same expansion gives
\begin{equation}\label{eq:cramer-numerator-approx}
 \sup_{D_{12}}|\mathcal C^p_{q,\nu}-\mathcal C^g_{q,\nu}|
 \le e^{-Bs_\nu},\qquad
 \sup_{D_{12}}|\mathcal C^g_{q,\nu}|\le e^{C_2s_\nu},
\end{equation}
after a possible reduction of $B$ by a constant depending only on $n+1,C$.
The constant $C_2$ is independent of $A$. Since $P_\nu$ is monic,
$\sup_{D_{12}}|P_\nu|\ge1$; hence \eqref{eq:wronskiapprox} ensures that
$R_\nu\not\equiv0$.

\step{Step 2. Comparison of the operators.}
For every monic polynomial $P$ of degree $m$,
\begin{equation}\label{eq:polynomial-negative-area}
 \int_{D_{12}}\log^-|P|\dA
 \le m\int_{|z|<1}\log\frac1{|z|}\dA=\frac{\pi m}2.
\end{equation}
Thus the set $E_\nu=\{z\in D_{12}:|P_\nu(z)|<e^{-s_\nu}\}$ satisfies
$\area(E_\nu)\le\pi m_\nu/(2s_\nu)\to0$. On its complement,
\eqref{eq:wronskiapprox} gives $|R_\nu|\ge\tfrac12e^{-s_\nu}$ once $B>1$.
If $b^p_{q,\nu}=\mathcal C^p_{q,\nu}/R_\nu$, Cramer's rule yields
\[
 |b^p_{q,\nu}-b_{q,\nu}|
 \le\frac{|\mathcal C^p_{q,\nu}-\mathcal C^g_{q,\nu}|}{|R_\nu|}
 +\frac{|\mathcal C^g_{q,\nu}|\,|R_\nu-P_\nu|}{|P_\nu R_\nu|}.
\]
By \eqref{eq:wronskiapprox} and \eqref{eq:cramer-numerator-approx}, this is
at most $2e^{-(B-1)s_\nu}+2e^{-(B-C_2-2)s_\nu}$ outside $E_\nu$.
Fix $A$ so that $B>C_2+2$. We have proved
\begin{equation}\label{eq:cramercomparison}
 b^p_{q,\nu}-b_{q,\nu}\convmeas0
 \quad\hbox{locally on }D_{12}.
\end{equation}

\step{Step 3. Separation of the roots.}
Write $P_\nu(z)=\prod_{i=1}^{m_\nu}(z-a_{\nu,i})$. Since
$\int_8^{10}\log^-|r-c|\,\mathrm dr\le2$ for $c\ge0$, there is
$\eta_\nu\in(8,10)$ such that
$\sum_i\log^-|\eta_\nu-|a_{\nu,i}||\le2m_\nu$.
Consequently $|P_\nu(z)|\ge e^{-2m_\nu}$ on $|z|=\eta_\nu$.
Since $m_\nu=o(s_\nu)$, we have
$e^{-Bs_\nu}<e^{-2m_\nu}$ for all sufficiently large $\nu$.
Thus \eqref{eq:wronskiapprox} and the preceding lower bound for
$|P_\nu|$ allow us to apply Rouch\'e's theorem on $|z|=\eta_\nu$,
obtaining
\begin{equation}\label{eq:rouche-root-count}
 n_{R_\nu}(\eta_\nu)=n_{P_\nu}(\eta_\nu)\le m_\nu=o(s_\nu).
\end{equation}
Here the circle contains no zero of either polynomial.

List the roots $\zeta_{\nu,r}$ of $R_\nu$ with multiplicity. Define
$U_{\nu,\mathrm{in}}(z)=\sum_{|\zeta_{\nu,r}|<\eta_\nu}|z-\zeta_{\nu,r}|^{-1}$
and define $U_{\nu,\mathrm{out}}$ by summing over the remaining roots.
The $L^{3/2}(D_6)$ norm of $|z-a|^{-1}$ is bounded uniformly for $a\in\C$.
Thus \eqref{eq:rouche-root-count}, Minkowski's inequality, and the degree
bound in \eqref{eq:wronskiapprox} imply
\begin{equation}\label{eq:rootbounds}
 \left\|U_{\nu,\mathrm{in}}/s_\nu\right\|_{L^{3/2}(D_6)}\longrightarrow0,
 \qquad \sup_{D_6}U_{\nu,\mathrm{out}}/s_\nu\le C_3.
\end{equation}
For the second estimate, every exterior root has distance greater than two
from $D_6$.

\step{Step 4. Holomorphic limits.}
Apply Proposition~\ref{prop:polynomialcoeff}, in the form
\eqref{eq:productformula}, to the span of the components of $\pv_\nu$.
After division by $s_\nu^q$, collect the terms using only exterior roots in
$B_{q,\nu}$, and the remaining terms in $F_{q,\nu}$. Then
\begin{equation}\label{eq:splitcoeff}
 b^p_{q,\nu}/s_\nu^q=B_{q,\nu}+F_{q,\nu},\qquad
 \sup_{D_6}|B_{q,\nu}|\le C_q,
\end{equation}
where $B_{q,\nu}$ is holomorphic on $D_8$. Estimating subsets by ordered
tuples in \eqref{eq:productformula} gives
\begin{equation}\label{eq:F-bound}
 |F_{q,\nu}|\le C_q\frac{U_{\nu,\mathrm{in}}}{s_\nu}
 \left(\frac{U_{\nu,\mathrm{in}}+U_{\nu,\mathrm{out}}}{s_\nu}\right)^{q-1}.
\end{equation}
The first factor on the right tends to zero in measure by
\eqref{eq:rootbounds}, and the parenthesized factor is bounded in measure.
Hence \eqref{eq:F-bound} implies $F_{q,\nu}\convmeas0$.
Montel's theorem applied to \eqref{eq:splitcoeff} gives a common subsequence
on which all $B_{q,\nu}$ converge locally uniformly on $D_6$.
Equation \eqref{eq:cramercomparison} transfers these limits to
$b_{q,\nu}/s_\nu^q$. All constants depend only on $n+1,C$ after the
choice of $A$ in Step~2, so the asserted bounds on $D_1$ follow as well.

\step{Step 5. The canonical gauge.}
Put $\ell_\nu=P_\nu'/((n+1)P_\nu)$. Its derivatives are sums of kernels
$(z-a_{\nu,i})^{-j-1}$. The estimates in Step~2 of the proof of
Lemma~\ref{lem:logderivlimit}, now with $m_\nu=o(s_\nu)$ terms, give
\begin{equation}\label{eq:gaugesmall}
 \ell_\nu^{(j)}/s_\nu^{j+1}\convmeas0,\qquad j\ge0.
\end{equation}
For completeness, the $j=0$ bound is
$\|\ell_\nu/s_\nu\|_{L^{3/2}(K)}\le C_Km_\nu/s_\nu$.
For $j\ge1$, choose $1/(j+1)<\alpha<\min\{1,2/(j+1)\}$; the integral of
$|\ell_\nu^{(j)}/s_\nu^{j+1}|^\alpha$ on $K\Subset D_4$ is at most
$C_{K,j}m_\nu/s_\nu^{(j+1)\alpha}$.

Conjugation by $P_\nu^{-1/(n+1)}$ replaces $D$ by $D+\ell_\nu$ in $L_\nu$.
The coefficient of $D^{n}$ becomes $b_{1,\nu}+(n+1)\ell_\nu=0$.
For each $q$, the difference $\widehat b_{q,\nu}-b_{q,\nu}$ is a finite
sum of monomials of weight $q$, assigning weight $r$ to $b_{r,\nu}$ and
weight $j+1$ to $\ell_\nu^{(j)}$. Each monomial contains at least one
$\ell_\nu^{(j)}$; this follows by expanding with $D\phi=\phi D+\phi'$.
The normalized $b_{r,\nu}$ are bounded in measure by Step~4, and
\eqref{eq:gaugesmall} makes each such monomial, divided by $s_\nu^q$, tend
to zero in measure. This proves \eqref{eq:gaugecomparison}.
\end{proof}

\subsection{A local representation of the rescaled curve}

We next construct a representation satisfying the hypotheses of
Proposition~\ref{prop:localcompact}. It is important that the bounds depend
only on the characteristic bound, not on the particular sequence of radii.

\begin{prop}\label{prop:representation}
Let $\fv=(f_0,\ldots,f_n)$ be a fixed reduced representation of a
linearly non-degenerate curve $f\colon\C\to\PP^n$, with
$f_j(0)\ne0$ for $0\leq j\leq n$.
Let $C>0$ be fixed, and suppose that $t_\nu,s_\nu\to\infty$ satisfy
\begin{equation}\label{eq:scalehyp}
 \log t_\nu=o(s_\nu),\qquad T(256t_\nu,f)\le Cs_\nu,
 \qquad N_{W_{\fv}}(256t_\nu)=o(s_\nu).
\end{equation}
After omitting finitely many terms, there are holomorphic functions
$H_\nu$ on $D_{64}$ such that
$\gv_\nu(z)=e^{-H_\nu(z)}\fv(t_\nu z)$ has a monic polynomial Wronskian
$P_\nu$, all of whose roots lie in $D_{64}$, and
\begin{equation}\label{eq:representationbounds}
 \deg P_\nu=o(s_\nu),\qquad
 \sup_{D_{32}}\log\norm{\gv_\nu}\le(5C+1)s_\nu,\qquad
 s_\nu^{-1}\log|g_{\nu,j}(0)|\longrightarrow0.
\end{equation}
Moreover, for $0<R<64$,
\begin{equation}\label{eq:meanidentity}
 \mean R{\log\norm{\gv_\nu}}=T(Rt_\nu,f)+c_\nu,
 \qquad c_\nu=o(s_\nu),
\end{equation}
where $c_\nu$ is independent of $R$. The sequences
\begin{equation}\label{eq:scaledcompact}
 (t_\nu/s_\nu)^qQ_q(t_\nu z),\qquad 2\le q\le n+1,
\end{equation}
are simultaneously relatively compact locally in measure on $D_4$, with
holomorphic limits bounded on $D_1$ by a constant depending only on $n+1,C$.
\end{prop}
\begin{proof}
\step{Step 1. The Wronskian and the value at the origin.}
Let $P_\nu$ be the monic polynomial with zeros $a/t_\nu$, where $a$ runs
through the zeros of $W_{\fv}$ in $D_{64t_\nu}$, with multiplicity.
By \eqref{eq:countbound} and \eqref{eq:scalehyp}, its degree
$m_\nu$ is $o(s_\nu)$. The quotient
$t_\nu^\kappa W_{\fv}(t_\nu z)/P_\nu(z)$ is holomorphic and zero-free on
$D_{64}$. Choose $H_\nu$ so that
\begin{equation}\label{eq:localgauge}
 e^{(n+1)H_\nu(z)}=\frac{t_\nu^\kappa W_{\fv}(t_\nu z)}{P_\nu(z)},
 \qquad \gv_\nu(z)=e^{-H_\nu(z)}\fv(t_\nu z).
\end{equation}
Formula \eqref{eq:wronskian-gauge} and the derivative scaling give
$W(\gv_\nu)=P_\nu$.

Write $W_{\fv}(z)=w_0z^{m_0}+O(z^{m_0+1})$, $w_0\ne0$.
Evaluating the zero-free quotient in \eqref{eq:localgauge} at the origin
and using \eqref{eq:zerojensen} gives
\begin{equation}\label{eq:centeridentity}
 (n+1)\Ree H_\nu(0)=\kappa\log t_\nu+N_{W_{\fv}}(64t_\nu)
                  -m_\nu\log64+\log|w_0|=o(s_\nu).
\end{equation}
Since every $f_j(0)$ is fixed and nonzero, \eqref{eq:centeridentity} proves the center-value
assertion in \eqref{eq:representationbounds}. Taking circular means in
\eqref{eq:localgauge} proves \eqref{eq:meanidentity}, with
$c_\nu=\log\norm{\fv(0)}-\Ree H_\nu(0)$.

\step{Step 2. The negative part of the norm.}
Partition $D_{64}$ into measurable sets $E_{\nu,j}$ on which
$|g_{\nu,j}|=\max_\ell|g_{\nu,\ell}|$, resolving ties by the smallest
index. Fix $j$ and put $h_\ell=g_{\nu,\ell}/g_{\nu,j}$; thus
$h_j=1$ and $|h_\ell|\le1$ on $E_{\nu,j}$. Expanding the Wronskian along
the constant column and factoring the other $h_\ell$ gives
\begin{equation}\label{eq:negativedeterminant}
 \frac{|P_\nu|}{\norm{\gv_\nu}^{n+1}}
 \le\left|\det\left(\frac{h_\ell^{(i)}}{h_\ell}\right)_{
                      1\le i\le n,\ \ell\ne j}\right|
 \quad\hbox{on }E_{\nu,j},
\end{equation}
apart from the isolated zeros at which a quotient is undefined.

For $\ell\ne j$, apply \eqref{eq:logderivbound} to
$h_\ell(48z)=f_\ell(48t_\nu z)/f_j(48t_\nu z)$.
Its value at zero is fixed and nonzero. By \eqref{eq:quotientbound}, its
characteristic at radius four is at most $T(192t_\nu,f)+O(1)\le Cs_\nu+O(1)$.
Thus, for $1\le i\le n$,
$m(48,h_\ell^{(i)}/h_\ell)=O(\log s_\nu+1)=o(s_\nu)$.
Expanding the determinant in \eqref{eq:negativedeterminant} proves that the
circular mean of its positive logarithm is $o(s_\nu)$, uniformly in $j$.
Also, integration of the logarithmic singularity on the circle gives
$\sup_{a\in\C}\mean{48}{\log^-|z-a|}<\infty$. Therefore
$\mean{48}{\log^-|P_\nu|}=O(m_\nu)=o(s_\nu)$.
Taking negative logarithms in \eqref{eq:negativedeterminant} and summing the
integrals over the sets $E_{\nu,j}$ yields
\begin{equation}\label{eq:negativepart}
 \mean{48}{\log^-\norm{\gv_\nu}}=o(s_\nu).
\end{equation}

\step{Step 3. The upper bound and the coefficients.}
Equations \eqref{eq:meanidentity} and \eqref{eq:negativepart} imply
$\mean{48}{\log^+\norm{\gv_\nu}}\le Cs_\nu+o(s_\nu)$.
The Poisson kernel of $D_{48}$ is bounded by five on $D_{32}$, so
subharmonicity gives the upper bound in \eqref{eq:representationbounds}.
Proposition~\ref{prop:localcompact} now applies to $\gv_\nu$.
Locally away from the Wronskian zeros, \eqref{eq:localgauge} gives, up to a
constant $(n+1)$st root of unity,
$P_\nu^{-1/(n+1)}\gv_\nu=t_\nu^{-\kappa/(n+1)}
(W_{\fv}^{-1/(n+1)}\fv)(t_\nu z)$.
The constant factor does not affect the operator. Hence the normalized
coefficients are
\begin{equation}\label{eq:canonical-coefficient-identification}
 \widehat b_{q,\nu}(z)=t_\nu^qQ_q(t_\nu z),\qquad 2\le q\le n+1,
\end{equation}
by \eqref{eq:scalecoeff}. Their compactness and bounds are precisely the
assertions about \eqref{eq:scaledcompact}.
\end{proof}

\subsection{Polynomial approximants used later}

The proof of regular variation also needs the polynomial solution spaces,
not just the limits of their differential coefficients. We state the
relevant estimates explicitly.

\begin{lem}\label{lem:replacement}
Let $\fv,t_\nu,s_\nu$ satisfy the hypotheses of
Proposition~\ref{prop:representation}. Let $H_\nu,\gv_\nu,P_\nu$ be the
functions, vectors, and monic polynomials constructed in Step~1 of its
proof, so that $\gv_\nu=e^{-H_\nu}\fv(t_\nu\cdot)$ and
$W(\gv_\nu)=P_\nu$. Fix a sufficiently large $A>0$.
Apply the Taylor-polynomial construction in Step~1 of the proof of
Proposition~\ref{prop:localcompact} to these vectors $\gv_\nu$, choosing
the truncation constant as in Step~2 of that proof. Denote the resulting
vectors by $\pv_\nu=(p_{\nu,0},\ldots,p_{\nu,n})$, and put
$R_\nu=W(\pv_\nu)$. Let $\eta_\nu\in(8,10)$ be the separating radii
chosen for these $P_\nu,R_\nu$ in Step~3 of the same proof.

The vectors $\pv_\nu$ have degree $O(s_\nu)$ and satisfy
\eqref{eq:taylorerror}. Their Wronskians satisfy
\begin{equation}\label{eq:replacement-data}
 \deg R_\nu=O(s_\nu),\qquad n_{R_\nu}(\eta_\nu)=o(s_\nu),\qquad
 s_\nu^{-1}\log|R_\nu|\longrightarrow0\text{ in }L^1_{\loc}(D_4).
\end{equation}
Let $U_{\nu,\mathrm{in}}(z)$ and $U_{\nu,\mathrm{out}}(z)$ be the sums
of $|z-\zeta|^{-1}$ over the zeros $\zeta$ of $R_\nu$ with
$|\zeta|<\eta_\nu$ and $|\zeta|>\eta_\nu$, respectively, counting
multiplicities. These are the functions constructed in Step~3 of the
proof of Proposition~\ref{prop:localcompact}, and they obey
\eqref{eq:rootbounds}. If
$L_{\pv_\nu}=D^{n+1}+\sum_{q=1}^{n+1}b^p_{q,\nu}D^{n+1-q}$, then
\begin{equation}\label{eq:polynomial-coefficients}
 \frac{b^p_{1,\nu}}{s_\nu}\convmeas0,\qquad
 \frac{b^p_{q,\nu}}{s_\nu^q}-(t_\nu/s_\nu)^qQ_q(t_\nu z)\convmeas0,
 \quad 2\le q\le n+1,
\end{equation}
locally on $D_4$.
\end{lem}
\begin{proof}
Steps~1 and~3 of the proof of Proposition~\ref{prop:localcompact},
applied as specified in the statement, give the approximation, degree
bound, and separating circle, including \eqref{eq:rootbounds}. Since all roots of $P_\nu$ lie in $D_{64}$ and
$\deg P_\nu=o(s_\nu)$, the bound
$\log^+|P_\nu(z)|\le(\deg P_\nu)\log(R+64)$ on $D_R$, together with
\eqref{eq:polynomial-negative-area}, gives
\begin{equation}\label{eq:small-polynomial-log}
 s_\nu^{-1}\log|P_\nu|\longrightarrow0
 \quad\hbox{in }L^1_{\loc}(\C).
\end{equation}
The area estimate holds on $D_R$ with the same bound $\pi\deg P_\nu/2$.

Fix $0<\epsilon<B$, where $B$ is as in \eqref{eq:wronskiapprox}.
Outside $\{|P_\nu|<e^{-\epsilon s_\nu}\}\cap D_{12}$, whose area tends to
zero by \eqref{eq:polynomial-negative-area}, we have
$|R_\nu/P_\nu-1|\le e^{-(B-\epsilon)s_\nu}$.
Together with \eqref{eq:small-polynomial-log}, this proves
$s_\nu^{-1}\log|R_\nu|\convmeas0$ on $D_4$.
These logarithms are locally uniformly bounded above by
\eqref{eq:wronskiapprox}, so Lemma~\ref{lem:subharmonic-compactness} upgrades
the convergence to $L^1_{\loc}$. This proves \eqref{eq:replacement-data}.
Finally, \eqref{eq:cramercomparison}, \eqref{eq:gaugecomparison}, and
\eqref{eq:canonical-coefficient-identification} give
\eqref{eq:polynomial-coefficients}, including the case $q=1$.
\end{proof}

\section{Growth indices}\label{sec:indices}

We use the full interval of P\'olya peak orders. The exact peak inequality
makes it possible to identify a peak order without introducing an additional
error in the exponent.

\subsection{Drasin--Shea indices and P\'olya peaks}

For a positive, unbounded, nondecreasing function $T$ on $(0,\infty)$,
its Drasin--Shea
indices \cite{DS72} are
\begin{equation}\label{eq:strong-indices}
 \begin{split}
 \rho^*(T)&=\sup\left\{p\in\R:
       \limsup_{x,t\to\infty}\frac{T(xt)}{x^pT(t)}=\infty\right\},\\
 \rho_*(T)&=\inf\left\{p\in\R:
       \liminf_{x,t\to\infty}\frac{T(xt)}{x^pT(t)}=0\right\}.
 \end{split}
\end{equation}
The limits are joint limits in $x,t$. For the characteristic of a
nonconstant holomorphic curve,
\begin{equation}\label{eq:index-order-bounds}
 0\le\rho_*(T)\le\lambda(f)\le\rho(f)\le\rho^*(T).
\end{equation}
We use the P\'olya peak theorem of Drasin--Shea \cite{DS72} in the form
stated in \cite[p.~1203]{Ere93}.

\begin{lem}[{\cite[p.~1203]{Ere93}}]\label{lem:peaks}
Let $T\colon(0,\infty)\to(0,\infty)$ be unbounded and nondecreasing.
For every finite
$\mu>0$ in $[\rho_*(T),\rho^*(T)]$, there are $r_\nu\to\infty$ and
$\epsilon_\nu\downarrow0$ such that
\begin{equation}\label{eq:peak}
 T(tr_\nu)\le(1+\epsilon_\nu)t^\mu T(r_\nu),
 \qquad \epsilon_\nu\le t\le\epsilon_\nu^{-1}.
\end{equation}
\end{lem}

In particular, for every fixed $t>0$, inequality \eqref{eq:peak} gives
$\limsup_\nu T(tr_\nu)/T(r_\nu)\le t^\mu$. The endpoints of the index
interval are included whenever they are finite and positive.

\subsection{Identification of the indices}

\begin{prop}\label{prop:indices}
Let $f$ be a transcendental linearly non-degenerate holomorphic curve of
finite lower order with $N_1(r,f)=o(T(r,f))$. Then
\begin{equation}\label{eq:indices-collapse}
 \rho_*(T)=\lambda(f)=\rho(f)=\rho^*(T)\in\{0\}\cup\Gn,
 \qquad T(r)=T(r,f).
\end{equation}
\end{prop}
\begin{proof}
If $\rho^*(T)=0$, the conclusion follows from \eqref{eq:index-order-bounds}.
Otherwise, fix a finite $\mu>0$ in the index interval, choose the peaks in
Lemma~\ref{lem:peaks}, and put $S_\nu=T(r_\nu)$. Transcendence gives
$\log r_\nu=o(S_\nu)$.

\step{Step 1. Limits on all compact subsets.}
For a fixed $R>0$, apply Proposition~\ref{prop:representation} with
$t_\nu=Rr_\nu$ and normalization $R^\mu S_\nu$. Inequality
\eqref{eq:peak} gives, for all sufficiently large $\nu$,
\begin{equation}\label{eq:peak-characteristic-scale}
 T(256Rr_\nu)\le2\,256^\mu R^\mu S_\nu.
\end{equation}
The constant is independent of $R$. The remaining conditions in
\eqref{eq:scalehyp} follow from $\log r_\nu=o(S_\nu)$ and
$N_1=o(T)$. For $2\le q\le n+1$, put
$A_{q,\nu}(z)=(r_\nu/S_\nu)^qQ_q(r_\nu z)$. The coefficient at scale
$Rr_\nu$ is related to it by
\begin{equation}\label{eq:peak-rescaling-relation}
 A_{q,\nu}(Rz)=R^{q(\mu-1)}
     \left(\frac{Rr_\nu}{R^\mu S_\nu}\right)^qQ_q(Rr_\nu z).
\end{equation}
Applying compactness for $R=1,2,\ldots$ and taking a diagonal subsequence
gives entire functions $a_q$ with $A_{q,\nu}\convmeas a_q$ on $\C$.
Uniqueness of limits in measure makes the limits agree on overlapping
disks. For any fixed $R>0$, apply the same compactness assertion to this
subsequence at scale $R$, and use uniqueness once more. Its uniform bound,
\eqref{eq:peak-characteristic-scale}, and
\eqref{eq:peak-rescaling-relation} give
\begin{equation}\label{eq:peak-coefficient-bound}
 \sup_{|z|\le R}|a_q(z)|\le K_{n,\mu}R^{q(\mu-1)},
 \qquad R>0.
\end{equation}

\step{Step 2. The limits do not all vanish.}
Suppose that every $a_q$ is zero. Apply
Proposition~\ref{prop:representation} at $t_\nu=r_\nu$, $s_\nu=S_\nu$.
Let $\gv_\nu$ and $P_\nu$ be the representation and Wronskian constructed
in Step~1 of its proof, and let $b_{q,\nu}$ be the coefficients of
$L_{\gv_\nu}$ from Lemma~\ref{lem:fundamental-operator}. The center values and upper bounds in \eqref{eq:representationbounds},
together with Lemma~\ref{lem:subharmonic-compactness}, give a further
subsequence with
$S_\nu^{-1}\log|g_{\nu,j}|\to u_j$ in $L^1_{\loc}(D_4)$, where the
$u_j$ are subharmonic limits not identically $-\infty$. Equations
\eqref{eq:gaugecomparison} and \eqref{eq:canonical-coefficient-identification}
imply $b_{q,\nu}/S_\nu^q\convmeas0$ for every $q$.
Divide the equation for $g_{\nu,j}$ by $S_\nu^{n+1}g_{\nu,j}$:
\begin{equation}\label{eq:normalized-peak-equation}
 \frac{g_{\nu,j}^{(n+1)}}{S_\nu^{n+1}g_{\nu,j}}
 +\sum_{q=1}^{n+1}\frac{b_{q,\nu}}{S_\nu^q}
       \frac{g_{\nu,j}^{(n+1-q)}}{S_\nu^{n+1-q}g_{\nu,j}}=0.
\end{equation}
Lemma~\ref{lem:logderivlimit} permits passage to the limit in
\eqref{eq:normalized-peak-equation}, giving $(2\partial u_j)^{n+1}=0$ almost
everywhere. Its Sobolev conclusion then makes every $u_j$ constant.
By \eqref{eq:norm-max}, $S_\nu^{-1}\log\norm{\gv_\nu}$ converges in
$L^1_{\loc}(D_4)$ to a constant $c$.

Set $F_\nu(t)=T(tr_\nu)/S_\nu$. Integrating \eqref{eq:meanidentity} over
annuli gives
\begin{equation}\label{eq:annular-characteristic-limit}
 \frac{2}{b^2-a^2}\int_a^b F_\nu(t)t\,\mathrm dt\longrightarrow c,
 \qquad 0<a<b<4.
\end{equation}
If $0<b<1$ and $a=b/2$, monotonicity and \eqref{eq:peak} imply
$c\le\limsup_\nu F_\nu(b)\le b^\mu$. Letting $b\downarrow0$ gives $c\le0$.
If $2<a<b<3$, however, $F_\nu(t)\ge F_\nu(1)=1$, and
\eqref{eq:annular-characteristic-limit} gives $c\ge1$. This contradiction
proves that some $a_q$ is nonzero.

\step{Step 3. The peak order and the index interval.}
Write $a_q(z)=\sum_{m\ge0}c_{q,m}z^m$. Cauchy's estimate applied to
\eqref{eq:peak-coefficient-bound} gives
$|c_{q,m}|\le K_{n,\mu}R^{q(\mu-1)-m}$ for every $R>0$.
Letting $R$ tend to zero or infinity shows that $c_{q,m}$ can be nonzero
only when $m=q(\mu-1)$. By Step~2, this occurs for some $q$, so
$q(\mu-1)$ is a nonnegative integer and $\mu\in\Gn$.

We have proved that every finite positive point of
$[\rho_*(T),\rho^*(T)]$ belongs to $\Gn$. The set $\Gn$ is locally finite:
on a bounded interval there are only finitely many choices of $q$ and $k$
in \eqref{ord}. The index interval therefore cannot contain two distinct
finite positive points. Its lower endpoint is finite by
\eqref{eq:index-order-bounds} and the finite lower order assumption, so its
upper endpoint cannot be infinite. The interval is a singleton, and
\eqref{eq:index-order-bounds} proves \eqref{eq:indices-collapse}.
\end{proof}

Section~\ref{sec:smallorder} will exclude the value zero. Before turning to
that case, we state the consequence of equal indices needed for arbitrary
sequences of radii.

\subsection{Uniform power bounds}

\begin{lem}\label{lem:power-bounds}
Let $T\colon(0,\infty)\to(0,\infty)$ be unbounded and nondecreasing, with
$\rho_*(T)=\rho^*(T)=\rho\in(0,\infty)$. For every $0<\epsilon<\rho$,
there are $C_\epsilon\ge1$ and $r_\epsilon>0$ such that
\begin{equation}\label{eq:power-bounds}
 C_\epsilon^{-1}\min\{t^{\rho-\epsilon},t^{\rho+\epsilon}\}
 \le\frac{T(tr)}{T(r)}
 \le C_\epsilon\max\{t^{\rho-\epsilon},t^{\rho+\epsilon}\},
 \qquad r,tr\ge r_\epsilon.
\end{equation}
\end{lem}
\begin{proof}
The definitions in \eqref{eq:strong-indices}, applied to
$p_-=\rho-\epsilon$ and $p_+=\rho+\epsilon$, give constants $c,C>0$ and
$X,r_0\ge1$ such that
$c x^{p_-}\le T(xr)/T(r)\le Cx^{p_+}$ for $x\ge X$ and $r\ge r_0$.
For $1\le x\le X$, monotonicity gives
$1\le T(xr)/T(r)\le T(Xr)/T(r)\le CX^{p_+}$.
Enlarging the constants therefore gives the desired bounds for all $x\ge1$.
For $0<x<1$, apply those bounds at the base radius $xr$ with multiplier
$1/x$ and invert. This is exactly \eqref{eq:power-bounds}.
\end{proof}

\section{The case of order less than one}\label{sec:smallorder}

The remaining possibility in Proposition~\ref{prop:indices} is order zero.
The argument below works more generally for curves of order less than one
and does not use P\'olya peaks. Its starting point is an entire-function
version of Proposition~\ref{prop:initial-basis}.

\subsection{An entire version of the initial-value estimate}

The order of an entire function $h$ is
$\limsup_{r\to\infty}\log^+\log^+M(r,h)/\log r$, with constants assigned
order zero.

\begin{lem}\label{lem:entire-majorant}
Let $\yv=(y_0,\ldots,y_n)$ be a vector of entire functions of order less
than one, normalized by $y_j^{(i)}(0)=\delta_{ij}$ for $0\le i,j\le n$.
If $(a_\ell)$ are the zeros of $W(\yv)$, repeated with multiplicity, then
$\sum_\ell|a_\ell|^{-1}<\infty$ and
\begin{equation}\label{eq:entire-majorant}
 |y_j(z)|\le\frac{|z|^j}{j!}
        \prod_\ell\left(1+\frac{|z|}{|a_\ell|}\right),
 \qquad z\in\C,\quad 0\le j\le n.
\end{equation}
\end{lem}
\begin{proof}
\step{Step 1. Uniform estimates for Taylor polynomials.}
Choose $\sigma\in(0,1)$ larger than the orders of all components, so that
$\log M(t,y_j)\le Ct^\sigma$ for $t\ge1$. Let $p_{j,N}$ be the Taylor
polynomial of degree $N\geq n$ and set
$W_N=W(p_{0,N},\ldots,p_{n,N})$. Cauchy's coefficient estimate gives
$M(t,p_{j,N})\le2M(2t,y_j)$. Applying Cauchy's derivative estimate on
$D_{2t}$ gives
$M(t,p_{j,N}^{(k)})\le C_kt^{-k}M(4t,y_j)$ for $0\le k\le n$.
The determinant expansion therefore implies
\begin{equation}\label{eq:WN-growth}
 W_N(0)=1,\qquad \log M(t,W_N)\le C_1t^\sigma\quad(t\ge1),
\end{equation}
with $C_1$ independent of $N$.

\step{Step 2. Control of roots escaping to infinity.}
Jensen's formula and \eqref{eq:WN-growth} imply
$n_{W_N}(t)\le C_2t^\sigma$ for $t\ge1$. If $a_{N,\ell}$ are the roots
of $W_N$, integration by parts gives
\begin{equation}\label{eq:reciprocal-tail}
 \sum_{|a_{N,\ell}|>R}\frac1{|a_{N,\ell}|}
 \le\int_R^\infty\frac{n_{W_N}(t)}{t^2}\,\mathrm dt
 \le C_3R^{\sigma-1},\qquad R\ge1.
\end{equation}
The same estimates apply to $W(\yv)$ by local uniform convergence, and
$W(\yv)(0)=1$. In particular, its reciprocal root sum converges.

\step{Step 3. Passage to the limit.}
The initial derivatives are preserved by Taylor truncation, so
\eqref{eq:basismajorant} applies to $(p_{0,N},\ldots,p_{n,N})$ at zero.
For any $R$ whose circle contains no zero of $W(\yv)$, Rouch\'e's theorem
matches the roots of $W_N$ in $D_R$ with those of $W(\yv)$, with
multiplicity. For fixed $r$, the logarithm of the contribution from roots
outside $D_R$ is at most $C_3rR^{\sigma-1}$, by
\eqref{eq:reciprocal-tail} and $\log(1+x)\le x$.
First let $N\to\infty$ for fixed $R$, and then let $R\to\infty$.
Since $\sigma<1$, this proves \eqref{eq:entire-majorant}.
\end{proof}

To express the product estimate as an integral, we use Tonelli's theorem
in its general form.

\begin{lem}[{\cite[p.~67, Theorem~2.37(a)]{Fol99}}]\label{lem:Tonelli}
Let $(X,\mathcal A,\mu)$ and $(Y,\mathcal B,\nu)$ be $\sigma$-finite
measure spaces, and let $\Phi\colon X\times Y\to[0,\infty]$ be
$\mathcal A\otimes\mathcal B$-measurable. The functions
$F(x)=\int_Y\Phi(x,y)\,\mathrm d\nu(y)$ and
$G(y)=\int_X\Phi(x,y)\,\mathrm d\mu(x)$ are measurable, and
\[
 \int_X F\,\mathrm d\mu
 =\int_{X\times Y}\Phi\,\mathrm d(\mu\times\nu)
 =\int_Y G\,\mathrm d\nu.
\]
The common value may be infinite.
\end{lem}

For a normalized entire system, this gives the following integral estimate.

\begin{cor}\label{cor:convolution}
Let $\yv=(y_0,\ldots,y_n)$ satisfy the hypotheses of
Lemma~\ref{lem:entire-majorant}, and list its Wronskian zeros as
$(a_\ell)_{\ell\in\mathcal I}$, including multiplicities. Put
\[
 H_{\yv}(r)=\log\max_{0\leq j\leq n}M(r,y_j),\qquad
 N_{\yv}(r)=N_{W(\yv)}(r)
          =\sum_{\ell\in\mathcal I}\log^+\frac r{|a_\ell|}.
\]
Then, for every $r>0$,
\begin{equation}\label{eq:convolution}
 H_{\yv}(r)\leq n\log^+r+
     r\int_0^\infty\frac{N_{\yv}(t)}{(r+t)^2}\,\mathrm dt.
\end{equation}
\end{cor}
\begin{proof}
The normalization gives $W(\yv)(0)=1$, so no $a_\ell$ is zero.
For $a\ne0$, integration by parts gives
\begin{equation}\label{eq:kernel-identity}
 r\int_0^\infty\frac{\log^+(t/|a|)}{(r+t)^2}\,\mathrm dt
 =\log\left(1+\frac r{|a|}\right).
\end{equation}
Apply Lemma~\ref{lem:Tonelli} with $X=\mathcal I$ equipped with counting
measure, $Y=(0,\infty)$ equipped with Lebesgue measure, and
$\Phi(\ell,t)=r\log^+(t/|a_\ell|)/(r+t)^2$.
Both measure spaces are $\sigma$-finite, and $\Phi$ is nonnegative and
measurable. Summing \eqref{eq:kernel-identity} therefore gives
\[
 \sum_{\ell\in\mathcal I}\log\left(1+\frac r{|a_\ell|}\right)
 =r\int_0^\infty\frac{N_{\yv}(t)}{(r+t)^2}\,\mathrm dt.
\]
The sum is finite by the reciprocal-root bound in
Lemma~\ref{lem:entire-majorant}. Taking logarithms and the maximum over
$j$ in \eqref{eq:entire-majorant} now proves \eqref{eq:convolution}.
\end{proof}

\subsection{A choice of radii for the integral estimate}

We shall use \eqref{eq:convolution} at radii where the integral can be
compared with its value at the base radius.

\begin{lem}\label{lem:envelope}
Let $H\colon(0,\infty)\to[0,\infty)$ be continuous, nondecreasing,
and positive for all sufficiently large arguments.
If $H(r)=o(r^\alpha)$ for some $0<\alpha<1$, there are $r_\nu\to\infty$
such that
\begin{equation}\label{eq:envelope}
 r_\nu\int_0^\infty\frac{H(3t)}{(r_\nu+t)^2}\,\mathrm dt
 \le I_\alpha H(r_\nu),\qquad
 I_\alpha=\int_0^\infty\frac{\max\{1,(3u)^\alpha\}}{(1+u)^2}\,\mathrm du.
\end{equation}
The constants $I_\alpha$ are finite and tend to one as $\alpha\downarrow0$.
\end{lem}
\begin{proof}
Choose $R_\nu\to\infty$ so that $H(R_\nu)>0$. The function
$H(t)/t^\alpha$ attains a maximum on $[R_\nu,\infty)$, because it tends to
zero. Let $r_\nu$ be a maximizing point. For $t\ge r_\nu$,
$H(t)\le H(r_\nu)(t/r_\nu)^\alpha$, while monotonicity controls smaller
$t$. Thus
$H(3r_\nu u)\le H(r_\nu)\max\{1,(3u)^\alpha\}$ for all $u>0$.
The substitution $t=r_\nu u$ proves \eqref{eq:envelope}.
For $0<\alpha\le\alpha_0<1$, its integrand is bounded by the integrable
function $\max\{1,(3u)^{\alpha_0}\}/(1+u)^2$. Dominated convergence gives
$I_\alpha\to\int_0^\infty(1+u)^{-2}\,\mathrm du=1$.
\end{proof}

\subsection{Classification below order one}

A rational normal curve with the parametrization relevant here has the form
\begin{equation}\label{eq:rationalnormal}
 f(z)=[(1,z,\ldots,z^{n})A],\qquad A\in\operatorname{GL}_{n+1}(\C).
\end{equation}
Here $\operatorname{GL}_{n+1}(\C)$ denotes the invertible $(n+1)\times(n+1)$ complex
matrices. This is a statement about the projective curve, not about a particular
reduced representation.

The coordinate construction used here requires only order less than one.
We state it separately so that it can also be applied to the order-zero
problem without a small-ramification hypothesis.

\begin{lem}\label{lem:small-order-coordinates}
Let $f=[\fv]\colon\C\to\PP^n$ be a linearly non-degenerate
holomorphic curve with $\rho(f)<1$. Assume that
$\fv=(f_0,\ldots,f_n)$ is reduced and $f_0(0)\ne0$. There is an entire function $G$
such that the components of $\gv=e^{-G}\fv$ have order at most $\rho(f)$.
Choose $b\in\C$ with $W(\gv)(b)\ne0$, and put
\[
 J_b=(g_j^{(i)}(b))_{0\leq i,j\leq n},\qquad
 \yv(z)=\gv(b+z)J_b^{-1},\qquad B=|b|.
\]
Then $y_j^{(i)}(0)=\delta_{ij}$ for $0\leq i,j\leq n$, and every
component $y_j$ has order at most $\rho(f)$. With $H_{\yv}$ and
$N_{\yv}$ defined in Corollary~\ref{cor:convolution}, there are constants
$C,C_b$ such that
\begin{equation}\label{eq:small-order-T}
 T(s,f)\leq H_{\yv}(s+B)+C\qquad(s\geq1)
\end{equation}
and
\begin{equation}\label{eq:translated-count}
 N_{\yv}(t)\leq N_{W_{\fv}}(t+B)+C_b\qquad(t\geq1).
\end{equation}
\end{lem}
\begin{proof}
\step{Step 1. Entire coordinates of order less than one.}
For every $\rho(f)<\beta<1$, Jensen's formula and $f_0(0)\ne0$ give
$N_{f_0}(r)\le T(r,f)+O(1)=O(r^\beta)$ and
$n_{f_0}(r)=O(r^\beta)$. Therefore the genus-zero product
$P(z)=\prod_{f_0(a)=0}(1-z/a)$ converges, with zeros repeated according to
multiplicity. Splitting its roots at $|a|=r$ gives
\begin{equation}\label{eq:canonical-product-growth}
 \log M(r,P)\le n_{f_0}(r)\log2+N_{f_0}(r)
                     +r\sum_{|a|>r}|a|^{-1}=O(r^\beta).
\end{equation}
The tail estimate follows by integrating $n_{f_0}(t)/t^2$.

Write $f_0=Pe^G$ with $G$ entire and put $\gv=e^{-G}\fv$.
Each $g_j=P(f_j/f_0)$ is entire. The scalar characteristic inequality,
\eqref{eq:quotientbound}, and \eqref{eq:canonical-product-growth} imply
$\TN(r,g_j)=O(r^\beta)$. The Poisson estimate then gives
$\log M(r,g_j)\le3\TN(2r,g_j)=O(r^\beta)$.
The product $P$, and hence this representation, is independent of $\beta$.
Letting $\beta\downarrow\rho(f)$ shows that all $g_j$ have order at most
$\rho(f)$. The zero-free gauge preserves both the curve and the Wronskian
zero divisor.

\step{Step 2. Initial derivatives and translated zeros.}
Choose $b$ with $W(\gv)(b)\ne0$ and put $B=|b|$. Let
$J_b=(g_j^{(i)}(b))_{0\leq i,j\leq n}$ and set
$\yv(z)=\gv(b+z)J_b^{-1}$. The matrix $J_b$ is invertible and
$y_j^{(i)}(0)=\delta_{ij}$. Constant basis changes and translations preserve
order less than one, so Lemma~\ref{lem:entire-majorant} applies. Norm
equivalence under the basis change gives
\eqref{eq:small-order-T}.
Also $H_{\yv}\ge0$, because $y_0(0)=1$, and
$H_{\yv}(r)\ge\log r$ for $r\ge1$, because $y_1'(0)=1$.

The Wronskian $W(\gv)$ has order less than one, so its nonzero zeros
satisfy $\sum_a|a|^{-1}<\infty$. The zeros of $W(\yv)$ are $a-b$.
For $|a|>2B$,
\[
 \log^+\frac t{|a-b|}
 \le\log^+\frac{t+B}{|a|}+\log^+\frac{|a|}{|a-b|}
 \le\log^+\frac{t+B}{|a|}+\frac{2B}{|a|}.
\]
The error terms are summable. Each of the finitely many remaining nonzero
zeros contributes at most its term in $N_{W_{\fv}}(t+B)$ plus a constant.
A zero at zero is treated by its multiplicity times $\log(t+B)$; no zero
is translated to zero, by the choice of $b$. Hence
\eqref{eq:translated-count}.
\end{proof}

We now add the small-ramification hypothesis to this construction.

\begin{prop}\label{prop:small-order}
Let $f\colon\C\to\PP^n$ be a linearly non-degenerate holomorphic curve
with $\rho(f)<1$ and $N_1(r,f)=o(T(r,f))$. Then $f$ has the form
\eqref{eq:rationalnormal}.
\end{prop}
\begin{proof}
\step{Step 1. The normalized entire system.}
Let $G,\gv,b,B$, and $\yv$ be obtained by the construction in the proof of
Lemma~\ref{lem:small-order-coordinates}, and let $H_{\yv},N_{\yv}$ be
as in Corollary~\ref{cor:convolution} for this $\yv$.
From $N_1=o(T)$, \eqref{eq:small-order-T}, and
\eqref{eq:translated-count}, for every $\varepsilon>0$ we obtain
\begin{equation}\label{eq:small-order-divisor-majorant}
 N_{\yv}(t)\le\varepsilon H_{\yv}(3t)+C_\varepsilon,
 \qquad t>0.
\end{equation}
Here $t+2B\le3t$ for large $t$, and a constant covers the remaining bounded
interval.

\step{Step 2. Polynomial growth on a sequence of radii.}
Choose $\alpha\in(0,1)$ larger than the component orders. Then
$H_{\yv}(r)=o(r^\alpha)$, and Lemma~\ref{lem:envelope} supplies radii
$r_\nu$. Substituting \eqref{eq:small-order-divisor-majorant} into
\eqref{eq:convolution} and using \eqref{eq:envelope} gives
\begin{equation}\label{eq:absorption}
 (1-\varepsilon I_\alpha)H_{\yv}(r_\nu)
 \le n\log r_\nu+C_\varepsilon.
\end{equation}
The constant term uses $r_\nu\int_0^\infty(r_\nu+t)^{-2}\,\mathrm dt=1$.
Choose $0<\varepsilon I_\alpha<1/(n+1)$ and put
$\gamma=n/(1-\varepsilon I_\alpha)<n+1$.
Equation \eqref{eq:absorption} yields
$M(r_\nu,y_j)\le C r_\nu^\gamma$. For every integer $k\ge n+1$, Cauchy's
estimate gives
$|y_j^{(k)}(0)|/k!\le Cr_\nu^{\gamma-k}\to0$.
Thus each $y_j$ is a polynomial of degree at most $n$, and the initial
conditions force $y_j(z)=z^j/j!$. Undoing the translation and the basis
change gives \eqref{eq:rationalnormal}.
\end{proof}

\subsection{The Wronskian counting function at order zero}

The same integral estimate gives a quantitative result without the
assumption $N_1=o(T)$.

\begin{prop}\label{prop:zero-order-ramification}
For a transcendental linearly non-degenerate holomorphic curve $f$ of order
zero,
\begin{equation}\label{eq:zero-order-ramification}
 \limsup_{r\to\infty}\frac{N_1(r,f)}{T(r,f)}\ge1.
\end{equation}
\end{prop}
\begin{proof}
Suppose otherwise and choose $0<c<1$ with
$N_{W_{\fv}}(r)\le cT(r,f)$ for all sufficiently large $r$.
Apply Lemma~\ref{lem:small-order-coordinates} to this curve. Let
$G,\gv,b,B$, and $\yv$ be the objects constructed in its proof, and
use $H_{\yv},N_{\yv}$ from Corollary~\ref{cor:convolution} for this
normalized system. Since $\rho(f)=0$, every component of $\yv$ has order
zero. Equations
\eqref{eq:small-order-T} and \eqref{eq:translated-count} give
$N_{\yv}(t)\le cH_{\yv}(3t)+C$ for all $t>0$.
Choose $0<\alpha<1$ so small that $cI_\alpha<1$.
Since $H_{\yv}(r)=o(r^\alpha)$, Lemma~\ref{lem:envelope} and
\eqref{eq:convolution} imply
$(1-cI_\alpha)H_{\yv}(r_\nu)\le n\log r_\nu+C$ on a sequence tending
to infinity. Thus $M(r_\nu,y_j)\le C_1r_\nu^\Gamma$ with
$\Gamma=n/(1-cI_\alpha)<\infty$.
Cauchy's estimate makes every Taylor coefficient of integer degree greater
than $\Gamma$ vanish. All $y_j$ are polynomials, so $f$ is rational, a
contradiction. This proves \eqref{eq:zero-order-ramification} and
Corollary~\ref{cor:zero-ratio}.
\end{proof}

\section{Regular variation}\label{sec:regular}

We now assume that $N_1(r,f)=o(T(r,f))$ and that the two Drasin--Shea
indices coincide at a positive value $\rho\in\Gn$. Write $T(r)=T(r,f)$.
We shall prove that every sequence of normalized characteristics has a
subsequence converging to the same power function.

\subsection{Limits along arbitrary radii}\label{sec:arbitrary-limits}

Fix an arbitrary sequence $r_\nu\to\infty$, and put $s_\nu=T(r_\nu)$.
For $0<\epsilon<\rho$, set
$M_\epsilon(R)=\max\{R^{\rho-\epsilon},R^{\rho+\epsilon}\}$.
Lemma~\ref{lem:power-bounds} gives
\begin{equation}\label{eq:regular-scale-bound}
 T(256Rr_\nu)\leq C_\epsilon256^{\rho+\epsilon}
 M_\epsilon(R)s_\nu
\end{equation}
for each fixed $R>0$ and all sufficiently large $\nu$.
The constant is independent of $R$. The conditions
$\log(Rr_\nu)=o(M_\epsilon(R)s_\nu)$ and
$N_{W_{\fv}}(256Rr_\nu)=o(M_\epsilon(R)s_\nu)$ follow from
transcendence, the small-ramification hypothesis, and
\eqref{eq:regular-scale-bound}. Thus Proposition~\ref{prop:representation}
applies with scale $Rr_\nu$ and normalization $M_\epsilon(R)s_\nu$.

Define
\begin{equation}\label{eq:arbitrary-scaled-coefficients}
 A_{q,\nu}(z)=\left(\frac{r_\nu}{s_\nu}\right)^qQ_q(r_\nu z),
 \qquad 2\leq q\leq n+1.
\end{equation}
The rescaling argument in \eqref{eq:peak-rescaling-relation}, now with
$M_\epsilon(R)$ in place of $R^\mu$, and a diagonal extraction over
$R=1,2,\ldots$ give entire functions $a_q$ such that
\begin{equation}\label{eq:arbitrary-coefficients}
 A_{q,\nu}\convmeas a_q\quad\text{locally on }\C,
 \qquad
 \sup_{|z|\leq R}|a_q(z)|
 \leq K_\epsilon\max\{R^{q(\rho-1-\epsilon)},
                         R^{q(\rho-1+\epsilon)}\}.
\end{equation}
Here $K_\epsilon$ is independent of $R$. Once the subsequence has been
chosen, the bound in \eqref{eq:arbitrary-coefficients} holds for every
$\epsilon\in(0,\rho)$: apply Proposition~\ref{prop:representation} again at
a fixed scale and use uniqueness of the limit in measure.

Cauchy's coefficient estimate, with $R\downarrow0$ and $R\to\infty$,
shows that the coefficient of $z^m$ in $a_q$ vanishes unless
$m=q(\rho-1)$. To avoid nonintegral powers in expressions on $\C$, put
\begin{equation}\label{eq:monomial-coefficients}
 J_\rho=\{q\in\{2,\ldots,n+1\}:q(\rho-1)\in\Z_{\geq0}\},
 \qquad
 a_q(z)=
 \begin{cases}
 c_qz^{q(\rho-1)},&q\in J_\rho,\\
 0,&q\notin J_\rho.
 \end{cases}
\end{equation}
Indeed, if $m>q(\rho-1)$, choose $\epsilon$ so small that
$m>q(\rho-1+\epsilon)$ and let $R\to\infty$ in Cauchy's estimate.
If $m<q(\rho-1)$, choose $m<q(\rho-1-\epsilon)$ and let $R\downarrow0$.

Apply Proposition~\ref{prop:representation} with $t_\nu=r_\nu$ and
normalization $s_\nu$. Let $H_\nu,\gv_\nu$, and $P_\nu$ be the
holomorphic functions, local representations, and monic Wronskians
constructed in Step~1 of its proof. The upper bounds and center
values in \eqref{eq:representationbounds}, followed by
Lemma~\ref{lem:subharmonic-compactness}, give, on a further subsequence,
$s_\nu^{-1}\log|g_{\nu,j}|\to u_j$ in $L^1_{\loc}(D_4)$, where no
$u_j$ is identically $-\infty$. By \eqref{eq:norm-max},
\begin{equation}\label{eq:norm-limit}
 U_\nu:=s_\nu^{-1}\log\norm{\gv_\nu}
 \longrightarrow U:=\max_{0\leq j\leq n}u_j
 \quad\text{in }L^1_{\loc}(D_4).
\end{equation}
The convergence \eqref{eq:small-polynomial-log} and Lemma~\ref{lem:sum}
give $\sum_j u_j\geq0$ almost everywhere. In particular, $U\geq0$
almost everywhere, and hence everywhere for its subharmonic
representative. Equation~\eqref{eq:meanidentity} gives
\begin{equation}\label{eq:radial-mean}
 \mean t{U_\nu}=\frac{T(tr_\nu)}{T(r_\nu)}+o(1),
 \qquad 0<t<4,
\end{equation}
with an error independent of $t$.

We need some information on the behavior of $U$ at the origin.
We integrate \eqref{eq:radial-mean}
with weight $2\pi t\,\mathrm dt$ over $(0,R)$, where $R<1/2$.
For $tr_\nu\geq r_\epsilon$, use the upper bound
$C_\epsilon t^{\rho-\epsilon}$ from \eqref{eq:power-bounds}; on the
remaining interval use $T(tr_\nu)\leq T(r_\epsilon)$. Passing to the
limit by \eqref{eq:norm-limit} yields
$\int_{D_R}U\,\dA\leq C_\epsilon R^{2+\rho-\epsilon}$.
Since $U\geq0$, the submean inequality on $D(z,|z|)\subset D_{2|z|}$
then gives
\begin{equation}\label{eq:origin-bound}
 0\leq U(z)\leq C_\epsilon|z|^{\rho-\epsilon}
 \quad (|z|<1/4),\qquad U(0)=0.
\end{equation}
The assertion at $0$ follows by applying the submean inequality on
$D_R$ and letting $R\downarrow0$.

Apply Lemma~\ref{lem:replacement} to these same $\gv_\nu,P_\nu$,
with one fixed sufficiently large approximation exponent $A$. Let
$\pv_\nu$ be the Taylor polynomials specified in that lemma and let
$\eta_\nu$ be its separating radii.
Their normalized logarithmic norms have the same limit:
\begin{equation}\label{eq:polynomial-norm}
 s_\nu^{-1}\log\norm{\pv_\nu}\longrightarrow U
 \quad\text{in }L^1_{\loc}(D_4).
\end{equation}
To see this, fix $0<\eta<A$. By \eqref{eq:norm-limit} and $U\geq0$,
the set where $\norm{\gv_\nu}<e^{-\eta s_\nu}$ has area tending to zero
on each compact subset. On its complement,
\eqref{eq:taylorerror} gives
$\norm{\pv_\nu-\gv_\nu}/\norm{\gv_\nu}
 \leq\sqrt{n+1}\,e^{-(A-\eta)s_\nu}$.
Thus the two logarithmic norms, divided by $s_\nu$, have the same
limit in measure. Their local upper bounds and
Lemma~\ref{lem:subharmonic-compactness} give
\eqref{eq:polynomial-norm}.

Write $R_\nu=W(\pv_\nu)$ and let
$L_{\pv_\nu}=D^{n+1}+\sum_{q=1}^{n+1}b^p_{q,\nu}D^{n+1-q}$ be
its fundamental operator. Retain the functions
$U_{\nu,\mathrm{in}},U_{\nu,\mathrm{out}}$ specified in
Lemma~\ref{lem:replacement} for these $R_\nu,\eta_\nu$, and put
$B_\nu(a)=U_{\nu,\mathrm{in}}(a)+U_{\nu,\mathrm{out}}(a)$, the sum
of $|a-\zeta|^{-1}$ over all zeros of $R_\nu$, with multiplicity.
By \eqref{eq:replacement-data} and the reciprocal-root bounds in
Lemma~\ref{lem:replacement}, we can pass to a subsequence and choose a
set $E\subset D_2\setminus\{0\}$ of full measure such that
\begin{equation}\label{eq:good-centers}
 \frac{\log|R_\nu(a)|}{s_\nu}\longrightarrow0,
 \qquad
 \limsup_{\nu\to\infty}\frac{B_\nu(a)}{s_\nu}<\infty,
 \qquad a\in E.
\end{equation}
Here the contribution of the interior roots, divided by $s_\nu$,
tends to zero almost everywhere on a subsequence; the exterior
contribution is uniformly bounded. Removing the countable union of
the zero sets of $R_\nu$ ensures that $R_\nu(a)\ne0$ for all $a\in E$
and all $\nu$.

\subsection{A unitary change of basis at a point}

The next lemma chooses component limits whose sum vanishes at a
prescribed point. The limiting norm $U$ does not change.

\begin{lem}\label{lem:basis-at-point}
Let $r_\nu,s_\nu,\pv_\nu,R_\nu,U$, and $E$ be the objects constructed
in Section~\ref{sec:arbitrary-limits}. Thus $s_\nu=T(r_\nu)$,
$R_\nu=W(\pv_\nu)$,
$U$ is the limit in \eqref{eq:polynomial-norm}, and $E$ is the full-measure
set chosen in \eqref{eq:good-centers}, with $R_\nu(a)\ne0$ for every
$a\in E$ and every $\nu$. For each $a\in E$, there are a further
subsequence and constant unitary $(n+1)\times(n+1)$ matrices $V_\nu$
such that the components of $\hv_\nu=\pv_\nu V_\nu$ satisfy
$s_\nu^{-1}\log|h_{\nu,j}|\to v_j$ in $L^1_{\loc}(D_4)$, where the
$v_j$ are subharmonic and not identically $-\infty$, and
\begin{equation}\label{eq:point-balance}
 U=\max_{0\leq j\leq n}v_j,
 \qquad \sum_{j=0}^{n}v_j\geq0\ \text{almost everywhere},
 \qquad \sum_{j=0}^{n}v_j(a)=0.
\end{equation}
\end{lem}

\begin{proof}
\step{Step 1. Singular values of the initial derivative matrix.}
Set $J_\nu(a)=(s_\nu^{-i}p_{\nu,j}^{(i)}(a))_{0\leq i,j\leq n}$.
Choose $V_\nu$ by singular value decomposition so that the columns of
$J_\nu(a)V_\nu$ are orthogonal, and denote their positive lengths by
$\sigma_{\nu,j}$. Cauchy's estimate and the upper bound for $\pv_\nu$
give $\sigma_{\nu,j}\leq e^{Cs_\nu}$. Moreover,
$\prod_j\sigma_{\nu,j}=s_\nu^{-\kappa}|R_\nu(a)|$.
By \eqref{eq:good-centers}, the logarithm of this product is
$o(s_\nu)$. Each singular value therefore also has an exponential
lower bound. After extraction,
\begin{equation}\label{eq:singular-exponents}
 s_\nu^{-1}\log\sigma_{\nu,j}\longrightarrow\lambda_j\in\R,
 \qquad \sum_{j=0}^{n}\lambda_j=0.
\end{equation}

Put $\hv_\nu=\pv_\nu V_\nu$. Its norm equals that of $\pv_\nu$, so
\eqref{eq:polynomial-norm} is unchanged. No sequence
$s_\nu^{-1}\log|h_{\nu,j}|$ can tend locally uniformly to $-\infty$:
Cauchy's estimates at $a$ would force
$s_\nu^{-1}\log\sigma_{\nu,j}\to-\infty$, contrary to
\eqref{eq:singular-exponents}. Subharmonic compactness gives the
component limits $v_j$. Equation~\eqref{eq:norm-max} gives
$U=\max_jv_j$. Since $|W(\hv_\nu)|=|R_\nu|$,
Lemma~\ref{lem:sum} and \eqref{eq:replacement-data} give
$\sum_jv_j\geq0$ almost everywhere.

\step{Step 2. Identification of the values at the center.}
For $0\leq i\leq n$, the definition of $\sigma_{\nu,j}$ gives
$|h_{\nu,j}^{(i)}(a)|\leq s_\nu^i\sigma_{\nu,j}$.
The initial-value estimate \eqref{eq:initial-value-bound} therefore yields
\begin{equation}\label{eq:point-initial-upper}
 |h_{\nu,j}(a+z)|
 \leq\sigma_{\nu,j}
 \left(\sum_{i=0}^{n}\frac{(s_\nu|z|)^i}{i!}\right)
 e^{|z|B_\nu(a)}.
\end{equation}
Choose $C_a>\limsup_\nu B_\nu(a)/s_\nu$.
Taking normalized logarithms in \eqref{eq:point-initial-upper} and using
\eqref{eq:singular-exponents} gives
$v_j(a+z)\leq\lambda_j+C_a|z|$ almost everywhere near $0$.
The submean inequality extends this bound to the subharmonic
representative, and hence $v_j(a)\leq\lambda_j$.

For the reverse inequality, fix $r>0$ with
$\overline{D(a,2r)}\subset D_4$. Cauchy's estimate gives
\begin{equation}\label{eq:singular-value-cauchy}
 \sigma_{\nu,j}\leq M(r,h_{\nu,j}(a+\cdot))
 \left(\sum_{i=0}^{n}\frac{(i!)^2}{(s_\nu r)^{2i}}\right)^{1/2}.
\end{equation}
For fixed $r$, the logarithm of the last factor is $o(s_\nu)$.
Apply the upper-bound assertion of
Lemma~\ref{lem:subharmonic-compactness} to
\eqref{eq:singular-value-cauchy}. It gives
$\lambda_j\leq\sup_{|z-a|<2r}v_j(z)$.
Let $r\downarrow0$ and use upper semicontinuity to obtain
$\lambda_j\leq v_j(a)$. Thus $v_j(a)=\lambda_j$ for every $j$.
Summation and \eqref{eq:singular-exponents} prove
\eqref{eq:point-balance}.
\end{proof}

The matrices and the further subsequence may depend on $a$. The function
$U$, already fixed by \eqref{eq:norm-limit}, is independent of these choices.

\subsection{Homogeneity of the limiting norm}

We use the following convexity criterion of Bergkvist and Rullg{\aa}rd
\cite[Corollary~1]{BR02}. The derivative in its statement is the
almost-everywhere defined weak derivative.

\begin{lem}[{\cite[Corollary~1]{BR02}}]\label{lem:finite-gradient-convex}
Let $G\subset\C$ be an open convex set, and let $A\subset\C$ be finite.
If $v$ is subharmonic on $G$ and $2\partial v\in A$ almost everywhere,
then $v$ is convex on $G$.
\end{lem}

\begin{prop}\label{prop:homogeneity}
Let $f$, $\rho$, and the sequence $r_\nu$ satisfy the assumptions
of Section~\ref{sec:arbitrary-limits}, and let $U$ be the subharmonic
limit constructed there in \eqref{eq:norm-limit}. Then
\begin{equation}\label{eq:homogeneity}
 U(tz)=t^\rho U(z)\qquad(t>0,\ z,tz\in D_2).
\end{equation}
\end{prop}

\begin{proof}
\step{Step 1. An algebraic equation for the gradient.}
Fix $a\in E$, and let $v$ be any component limit supplied by
Lemma~\ref{lem:basis-at-point} at this $a$. Let $h_\nu$ be the
corresponding component of $\hv_\nu=\pv_\nu V_\nu$ constructed in
Step~1 of that lemma's proof. Since $V_\nu$ is constant in $z$,
$h_\nu$ is annihilated by the operator $L_{\pv_\nu}$ defined in
Section~\ref{sec:arbitrary-limits}. Dividing its
differential equation by $s_\nu^{n+1} h_\nu$ gives
\begin{equation}\label{eq:normalized-polynomial-equation}
 \frac{h_\nu^{(n+1)}}{s_\nu^{n+1} h_\nu}
 +\sum_{q=1}^{n+1}\frac{b_{q,\nu}^p}{s_\nu^q}
                  \frac{h_\nu^{(n+1-q)}}{s_\nu^{n+1-q}h_\nu}=0
\end{equation}
almost everywhere. By \eqref{eq:polynomial-coefficients},
\eqref{eq:arbitrary-scaled-coefficients},
\eqref{eq:arbitrary-coefficients}, and
\eqref{eq:monomial-coefficients}, the coefficient factors converge in
measure to $0$ for $q=1$ and to $a_q$ for $q\geq2$.
Lemma~\ref{lem:logderivlimit} permits passage to the limit in every term
of \eqref{eq:normalized-polynomial-equation}. Thus
\begin{equation}\label{eq:gradient-equation}
 (2\partial v)^{n+1}+
 \sum_{q\in J_\rho}c_qz^{q(\rho-1)}(2\partial v)^{n+1-q}=0
 \quad\text{almost everywhere on }D_4.
\end{equation}
The coefficients of this monic polynomial are bounded on compact
subsets of $D_4$. Its roots, and hence the weak gradient of $v$, are
locally bounded. Therefore $v$ has a locally Lipschitz representative,
which agrees with its subharmonic representative. Choosing one
$a_0\in E$ in \eqref{eq:point-balance} shows that $U$, as a maximum of
finitely many such functions, is also locally Lipschitz on $D_4$.
For every subsequent choice of $a$, both identities and the inequality
in \eqref{eq:point-balance} consequently hold pointwise.

\step{Step 2. Convexity after a conformal change of variable.}
Fix $a\in E$ and its component limits $v_0,\ldots,v_n$.
Choose a sufficiently narrow open sector $S\subset D_2$ with vertex $0$
and containing $a$, so that a branch of $w=z^\rho/\rho$ maps $S$
conformally onto a convex sector $G$. Set $V_j(w)=v_j(z(w))$.
These functions are subharmonic. The weak chain rule gives
$2\partial_zv_j=z^{\rho-1}2\partial_wV_j$.
On the chosen sector, substitution in \eqref{eq:gradient-equation}
and cancellation of $z^{(n+1)(\rho-1)}$ yield
\begin{equation}\label{eq:constant-gradient-equation}
 (2\partial_wV_j)^{n+1}+\sum_{q\in J_\rho}c_q(2\partial_wV_j)^{n+1-q}=0
 \quad\text{almost everywhere on }G.
\end{equation}
The roots of the fixed polynomial in
\eqref{eq:constant-gradient-equation} form a finite set.
Lemma~\ref{lem:finite-gradient-convex} shows that each $V_j$ is convex.

\step{Step 3. Equality on radial segments.}
Applying \eqref{eq:sandwich} to \eqref{eq:point-balance} gives
$-nU\leq v_j\leq U$. By \eqref{eq:origin-bound}, each $V_j$
therefore extends continuously to the vertex, with $V_j(0)=0$.
Put $w_0=a^\rho/\rho$. Since $V_j(0)=0$, convexity gives
\[
 V_j(tw_0)\leq tV_j(w_0),\qquad 0<t<1.
\]
On the other hand, \eqref{eq:point-balance} gives
\[
 0\leq\sum_jV_j(tw_0)
 \leq t\sum_jV_j(w_0)=0.
\]
Hence equality holds in each of the preceding convexity inequalities.

Now let $t=s^\rho$, where $0<s<1$. Since
\[
 \frac{(sa)^\rho}{\rho}=s^\rho w_0,
\]
taking the maximum over $j$ gives
\[
 U(sa)=s^\rho U(a),\qquad a\in E,\quad 0<s<1.
\]
Since $E$ is dense and $U$ is continuous, this identity extends to
all $a\in D_2$. The case $s>1$ follows by applying the identity with
multiplier $1/s$ at $sa$. This proves \eqref{eq:homogeneity}.
\end{proof}

\subsection{Convergence of the characteristic ratios}

Homogeneity identifies all possible limits of the normalized characteristic.

\begin{prop}\label{prop:regular-variation}
Let $f\colon\C\to\PP^n$ be a transcendental linearly non-degenerate
holomorphic curve, put $T(r)=T(r,f)$, and assume that $N_1(r,f)=o(T(r))$
and $\rho_*(T)=\rho^*(T)=\rho\in\Gn$. Then
\begin{equation}\label{eq:ratio-limit}
 \frac{T(cr)}{T(r)}\longrightarrow c^\rho\qquad(r\to\infty),
\end{equation}
locally uniformly for $c\in(0,\infty)$.
\end{prop}

\begin{proof}
\step{Step 1. Identification of a subsequential limit.}
Starting with an arbitrary $r_\nu\to\infty$, carry out the
construction and subsequence extractions in
Section~\ref{sec:arbitrary-limits}. Let $U_\nu,U$ be the functions
constructed there in \eqref{eq:norm-limit}, and put
$F_\nu(t)=T(tr_\nu)/T(r_\nu)$.
Integrating \eqref{eq:radial-mean} over an annulus and using
\eqref{eq:norm-limit} and the homogeneity \eqref{eq:homogeneity} from
Proposition~\ref{prop:homogeneity} gives
\begin{equation}\label{eq:annular-ratio-limit}
 \int_a^bF_\nu(t)t\,\dd t\longrightarrow
 K\int_a^bt^{\rho+1}\,\dd t,
 \qquad 0<a<b<2,\qquad K=\mean1U.
\end{equation}
Indeed, the limiting area integral is the integral of $U$, whose
circular mean at radius $t$ is $Kt^\rho$.

For $0<t-h<t<t+h<2$, monotonicity gives
\[
 \frac{\int_{t-h}^tF_\nu(s)s\,\dd s}{\int_{t-h}^ts\,\dd s}
 \leq F_\nu(t)\leq
 \frac{\int_t^{t+h}F_\nu(s)s\,\dd s}{\int_t^{t+h}s\,\dd s}.
\]
Apply \eqref{eq:annular-ratio-limit} and let $h\downarrow0$.
It follows that $F_\nu(t)\to Kt^\rho$ for $0<t<2$.
Since $F_\nu(1)=1$, we have $K=1$.

\step{Step 2. The full limit and uniformity.}
Every original sequence has a subsequence with this same limit.
The sequential criterion for convergence therefore gives
$T(tr)/T(r)\to t^\rho$ for each $t\in(0,2)$.
For $c\geq2$, choose an integer $k$ so that $b=c^{1/k}\in(1,2)$ and write
\[
 \frac{T(cr)}{T(r)}=
 \prod_{j=0}^{k-1}\frac{T(b^{j+1}r)}{T(b^jr)}.
\]
Each factor tends to $b^\rho$, proving \eqref{eq:ratio-limit} for every
$c>0$. Finally, on a compact interval $[a,b]\subset(0,\infty)$,
choose a finite partition on which the oscillation of $c^\rho$ on
each subinterval is small. Convergence at the partition points,
together with monotonicity of $T(cr)$ in $c$, bounds the error uniformly
between consecutive points. This proves local uniformity.
\end{proof}

We now complete the proof of Theorem~\ref{thm:main}.

\begin{proof}
Transcendence gives $\log r=o(T(r,f))$. Under \eqref{eq:mainhyp},
Proposition~\ref{prop:indices}
therefore implies that the lower order, order, and both Drasin--Shea
indices coincide at a value in $\{0\}\cup\Gn$.
Proposition~\ref{prop:small-order} excludes zero, since a curve satisfying
its hypotheses is rational. Thus the common value belongs to $\Gn$.
By Proposition~\ref{prop:regular-variation},
$\ell(r)=T(r,f)/r^\rho$ satisfies
$\ell(cr)/\ell(r)=c^{-\rho}T(cr,f)/T(r,f)\to1$, locally uniformly for
$c>0$. Hence $\ell$ is slowly varying and \eqref{reg} holds.
\end{proof}

\section{Sharpness}\label{sec:sharpness}

We give details for the examples mentioned in the Introduction. The first
construction is the classical polynomial-coefficient differential equation
used there. The second is an explicit curve attaining the order-zero bound.

\subsection{Realization of every allowed order}

The raywise asymptotics use an elementary result for a diagonal system
with an integrable perturbation. We include its proof.

\begin{lem}\label{lem:integrable-system}
Let $q\geq1$ be an integer, let $\Lambda=\operatorname{diag}(\lambda_1,\ldots,\lambda_q)$
with pairwise distinct real parts $\Ree\lambda_j$, and let
$E\colon[t_0,\infty)\to\C^{q\times q}$ be continuous, with
$\int_{t_0}^\infty\|E(t)\|\,\mathrm dt<\infty$. Here $\|\cdot\|$ is the
operator norm induced by the Euclidean norm. The system
$X'=(\Lambda+E(t))X$ has a fundamental system of solutions
\[
 X_j(t)=e^{\lambda_jt}(e_j+o(1)),\qquad 1\leq j\leq q,
\]
where $e_1,\ldots,e_q$ are the standard coordinate vectors in $\C^q$.
\end{lem}
\begin{proof}
Fix $j$ and seek $X_j(t)=e^{\lambda_jt}(e_j+u(t))$. For a sufficiently
large $T\geq t_0$, define an operator on bounded continuous vector
functions on $[T,\infty)$ by
\[
 (\mathcal T_j u)_i(t)=
 \begin{cases}
 \displaystyle\int_T^t e^{(\lambda_i-\lambda_j)(t-s)}
           [E(s)(e_j+u(s))]_i\,\mathrm ds,
           &\Ree\lambda_i<\Ree\lambda_j,\\[2mm]
 \displaystyle-\int_t^\infty e^{(\lambda_i-\lambda_j)(t-s)}
           [E(s)(e_j+u(s))]_i\,\mathrm ds,
           &\Ree\lambda_i\geq\Ree\lambda_j.
 \end{cases}
\]
The brackets with subscript $i$ denote the $i$th coordinate.
Every exponential kernel has modulus at most $1$ on its integration
interval. Thus the Lipschitz constant of $\mathcal T_j$ in the supremum
norm is at most $q\int_T^\infty\|E(s)\|\,\mathrm ds$, which is less than
$1$ for sufficiently large $T$. The contraction theorem gives a fixed
point $u$. Its backward integrals tend to zero by integrability. For each
forward integral, its part over a fixed bounded interval tends to zero
because $\Ree(\lambda_i-\lambda_j)<0$, while the remaining part is bounded
by an arbitrarily small integrable tail. Hence $u(t)\to0$.

Differentiation of the integral equations gives
$u'=(\Lambda-\lambda_j I)u+E(t)(e_j+u)$, so $X_j$ solves the system.
After the exponential factor is removed from each column, the matrix of
these $q$ solutions tends to the identity. Its determinant is therefore
nonzero for large $t$. Existence and uniqueness for a linear system extend
this fundamental system to $[t_0,\infty)$.
\end{proof}

We now apply this lemma to the polynomial-coefficient equation in the
Introduction.

\begin{prop}\label{prop:sharpness-orders}
Let $q$ and $k$ be integers with $2\leq q\leq n+1$ and $k\geq0$, and put
$m=n+1-q$, $\beta=k/q$, and $\rho=1+\beta$. Consider
\begin{equation}\label{eq:sharpness-equation}
 w^{(n+1)}=z^k w^{(m)}.
\end{equation}
Let $f_0,\ldots,f_n$ be its entire solutions with initial values
$f_j^{(i)}(0)=\delta_{ij}$ for $0\leq i,j\leq n$.
Then $\fv=(f_0,\ldots,f_n)$ is a reduced representation of a
transcendental linearly non-degenerate curve, $W(\fv)\equiv1$, and
\begin{equation}\label{eq:sharpness-asymptotic}
 T(r,\fv)\sim\frac{q\sin(\pi/q)}{\pi\rho}\,r^\rho.
\end{equation}
\end{prop}

\begin{proof}
Abel's identity gives $W(\fv)\equiv1$, because the coefficient of
$w^{(n)}$ in \eqref{eq:sharpness-equation} is zero and the initial
Wronskian is $1$. Thus the components are linearly independent and have
no common zero. The positive asymptotic constant in
\eqref{eq:sharpness-asymptotic} will imply transcendence.

\step{Step 1. A uniform bound on circles.}
For a fixed $R\geq1$, replace the derivative vector of a solution by
$(w,R^{-\beta}w',\ldots,R^{-n\beta}w^{(n)})$.
Equation~\eqref{eq:sharpness-equation} becomes a first-order system whose
matrix is $R^\beta$ times a matrix bounded on $D_R$ independently of $R$:
its superdiagonal entries are $1$, and its last row has the single
entry $(z/R)^k$ in position $m$ (positions are numbered from $0$).
Integration along a radial segment and Gronwall's inequality give
$|f_j^{(i)}(z)|\leq R^{i\beta}e^{CR^\rho}$ for $|z|\leq R$,
after increasing $C$ to cover the initial values.
Hadamard's determinant inequality, applied to $W(\fv)=1$, bounds the
norm of the row of derivative order zero from below by the reciprocal of the
product of the norms of the remaining rows. Hence
\begin{equation}\label{eq:sharpness-domination}
 \bigl|\log\norm{\fv(re^{i\theta})}\bigr|\leq C r^\rho
 \qquad(r\geq1,\ -\pi\leq\theta\leq\pi).
\end{equation}

\step{Step 2. Growth on a ray.}
The function $y=w^{(m)}$ satisfies $y^{(q)}=z^ky$.
On a fixed ray, choose branches of $z^\beta$ and $z^\rho$ and set
$Y=(y,z^{-\beta}y',\ldots,z^{-(q-1)\beta}y^{(q-1)})^{\mathsf T}$.
With $x=z^\rho/\rho$, its system is
\begin{equation}\label{eq:sharpness-scaled-system}
 \frac{\mathrm dY}{\mathrm dx}
 =\left(A-\frac{\beta}{\rho x}\operatorname{diag}(0,\ldots,q-1)\right)Y,
\end{equation}
where $A$ has superdiagonal entries $1$, lower-left entry $1$, and all
other entries zero. Its eigenvalues are the $q$th roots of unity, and
corresponding eigenvectors are $(1,\zeta,\ldots,\zeta^{q-1})^{\mathsf T}$.

Write $x=e^{i\phi}t$, where $\phi=\rho\theta$ and $t=r^\rho/\rho$.
Diagonalizing $A$ in \eqref{eq:sharpness-scaled-system} gives
$Z'=(\Lambda+B/t)Z$, where
$\Lambda=\operatorname{diag}(e^{i\phi}\zeta)$.
Every diagonal entry of $B$ equals
$b=-\beta(q-1)/(2\rho)$: conjugation by the Fourier matrix replaces
the diagonal of $\operatorname{diag}(0,\ldots,q-1)$ by its average.
For all but finitely many directions $\theta$, the real parts of the
entries $\lambda_\zeta=e^{i\phi}\zeta$ of $\Lambda$ are distinct.
Choose an off-diagonal matrix $C$ with
$C_{ij}=-B_{ij}/(\lambda_i-\lambda_j)$ and $C_{ii}=0$.
The substitution $Z=(I+C/t)t^bX$ then gives
\begin{equation}\label{eq:sharpness-integrable-system}
 X'=(\Lambda+E(t))X,\qquad E(t)=O(t^{-2}).
\end{equation}

Apply Lemma~\ref{lem:integrable-system} with the diagonal matrix
$\Lambda$ and perturbation $E(t)$ just obtained in
\eqref{eq:sharpness-integrable-system}. The real parts of the diagonal
entries are distinct in the chosen direction, and $E(t)=O(t^{-2})$ is
integrable on a sufficiently far tail of the ray. The lemma supplies a
fundamental system $X_j(t)=e^{\lambda_jt}(e_j+o(1))$ for this system.

Undoing the transformations gives a fundamental system of scalar solutions
$y_\zeta$ with
$y_\zeta(re^{i\theta})=t^b e^{\lambda_\zeta t}(1+o(1))$,
up to nonzero constant factors. Put
\begin{equation}\label{eq:sharpness-indicator}
 h(\theta)=\frac1\rho\max_{\zeta^q=1}\Re(\zeta e^{i\rho\theta}).
\end{equation}
The maximum is positive outside the finitely many exceptional directions.
If $m>0$, adjoining the polynomials $1,z,\ldots,z^{m-1}$ and taking
$m$ primitives of the $y_\zeta$ gives a fundamental system for
\eqref{eq:sharpness-equation}. Integration along the ray does not
increase any exponential rate beyond $h(\theta)$. For a solution with
the maximal rate, which is positive, each integration preserves that
rate: integration of $r^a e^{\lambda r^\rho/\rho}(1+o(1))$, with
$\Re\lambda>0$, gives
$\lambda^{-1}r^{a-\beta}e^{\lambda r^\rho/\rho}(1+o(1))$,
apart from the constant ray factor. This follows by integration by
parts for the leading term and an absolute bound for its $o(1)$ error.
A fixed nonsingular change of basis changes the logarithm of the norm
by $O(1)$. Therefore
\begin{equation}\label{eq:sharpness-ray-limit}
 r^{-\rho}\log\norm{\fv(re^{i\theta})}\longrightarrow h(\theta)
\end{equation}
outside those finitely many directions. The case $m=0$ follows without
integration.

\step{Step 3. Circular means.}
Equations~\eqref{eq:sharpness-domination} and
\eqref{eq:sharpness-ray-limit} permit dominated convergence in the
definition of $T$. The function
$\max_{\zeta^q=1}\Re(\zeta e^{i\phi})$ has period $2\pi/q$, and its
integral over one period is $2\sin(\pi/q)$.
As $\theta$ traverses an interval of length $2\pi$, the variable
$\phi=\rho\theta$ traverses exactly $q+k$ such periods.
Thus \eqref{eq:sharpness-indicator} gives
$\frac1{2\pi}\int_{-\pi}^{\pi}h(\theta)\,\mathrm d\theta
=q\sin(\pi/q)/(\pi\rho)$, proving
\eqref{eq:sharpness-asymptotic}.
\end{proof}

\subsection{Equality in the order-zero bound}

The following explicit example attains the lower bound.

\begin{prop}\label{prop:sharpness-zero}
Set
\begin{equation}\label{eq:zero-sharp-example}
 g(z)=\prod_{j=1}^\infty(1+e^{-j}z),
 \qquad \fv=(1,z,\ldots,z^{n-1},g).
\end{equation}
Then $\fv$ represents a transcendental linearly non-degenerate
holomorphic curve of order zero, and
\[
 \lim_{r\to\infty}\frac{N_1(r,\fv)}{T(r,\fv)}=1.
\]
\end{prop}
\begin{proof}
The product converges locally uniformly. If $x=\log r$ and $r\to\infty$,
then
\begin{equation}\label{eq:zero-sharp-growth}
 N_g(r)=\sum_{j\leq x}(x-j)=\tfrac12x^2+O(x),
 \qquad \log M(r,g)=N_g(r)+O(1).
\end{equation}
For the second estimate, the difference is
$\sum_{j\geq1}\log(1+e^{-|x-j|})$, which is uniformly bounded by two
geometric series. Jensen's formula and $g(0)=1$ give
$N_g(r)\leq m(r,g)\leq\log M(r,g)$. Thus $g$ has order zero and
$m(r,g)\sim\tfrac12(\log r)^2$.

We also need the zeros of $g^{(n)}$. Let
$g_N(z)=\prod_{j=1}^N(1+e^{-j}z)$. Its zeros are simple and negative.
Repeated applications of Rolle's theorem give, for $N\geq n$,
\begin{equation}\label{eq:zero-sharp-interlacing}
 \max\{0,n_{g_N}(r)-n\}\leq n_{g_N^{(n)}}(r)\leq n_{g_N}(r).
\end{equation}
Indeed, the $i$th zero of $g_N^{(n)}$, ordered by increasing absolute
value, lies between the $i$th and $(i+n)$th zeros of $g_N$ in that
ordering. Since $g_N^{(n)}\to g^{(n)}$ locally uniformly and
$g^{(n)}(0)>0$, Hurwitz's theorem passes
\eqref{eq:zero-sharp-interlacing} to the limit, first at radii avoiding
zeros and then by one-sided limits. It also shows that every zero of
$g^{(n)}$ is negative. Integrating the counting inequalities from a
fixed positive radius gives
$N_{g^{(n)}}(r)=N_g(r)+O(\log r)$.

The vector in \eqref{eq:zero-sharp-example} is reduced, linearly
non-degenerate, and transcendental. Direct expansion of its Wronskian gives
$W_{\fv}=(\prod_{j=0}^{n-1}j!)g^{(n)}$.
Moreover, comparison of its norm with $\max\{1,|g|\}$ gives
$T(r,\fv)=m(r,g)+O(\log r)$.
Together with \eqref{eq:zero-sharp-growth}, these identities imply
\[
 \lim_{r\to\infty}\frac{N_1(r,\fv)}{T(r,\fv)}=1.
\]
This proves that Corollary~\ref{cor:zero-ratio} is sharp.
\end{proof}


\begin{thebibliography}{Ere89a}
\raggedright

\bibitem[BR02]{BR02}
T.~Bergkvist and H.~Rullg\aa rd,
\emph{On polynomial eigenfunctions for a class of differential operators},
Math. Res. Lett. \textbf{9} (2002), no.~2--3, 153--171.
\doi{10.4310/MRL.2002.v9.n2.a3}.

\bibitem[Car33]{Car33}
H.~Cartan,
\emph{Sur les z\'eros des combinaisons lin\'eaires de $p$ fonctions
holomorphes donn\'ees},
Mathematica (Cluj) \textbf{7} (1933), 5--31.

\bibitem[Dra81]{Dra81}
D.~Drasin,
\emph{Quasi-conformal modifications of functions having deficiency sum two},
Ann.\ of Math. (2) \textbf{114} (1981), no.~3, 493--518.
\doi{10.2307/1971300}.

\bibitem[Dra87]{Dra87}
D.~Drasin,
\emph{Proof of a conjecture of F.~Nevanlinna concerning functions which
have deficiency sum two},
Acta Math. \textbf{158} (1987), no.~1--2, 1--94.
\doi{10.1007/BF02392256}.

\bibitem[DS72]{DS72}
D.~Drasin and D.~F.~Shea,
\emph{P\'olya peaks and the oscillation of positive functions},
Proc. Amer. Math. Soc. \textbf{34} (1972), no.~2, 403--411.
\doi{10.1090/S0002-9939-1972-0294580-X}.

\bibitem[Ere89a]{Ere89a}
A.~Eremenko,
\emph{A new proof of the Drasin theorem on meromorphic functions of finite
order with a maximal sum of defects.~I},
Teor. Funktsii Funktsional. Anal.\ i Prilozhen. No.~51 (1989), 107--116
(Russian); English translation, J. Soviet Math. \textbf{52} (1990),
no.~6, 3522--3529.

\bibitem[Ere89b]{Ere89b}
A.~Eremenko,
\emph{A new proof of the Drasin theorem on meromorphic functions of finite
order with a maximal sum of defects.~II},
Teor. Funktsii Funktsional. Anal.\ i Prilozhen. No.~52 (1989), 69--77
(Russian); English translation, J. Soviet Math. \textbf{52} (1990),
no.~5, 3397--3403.
\doi{10.1007/BF01099906}.

\bibitem[Ere93]{Ere93}
A.~Eremenko,
\emph{Meromorphic functions with small ramification},
Indiana Univ. Math. J. \textbf{42} (1993), no.~4, 1193--1218.
\doi{10.1512/iumj.1993.42.42055}.

\bibitem[Ere98]{Ere98}
A.~Eremenko,
\emph{Extremal holomorphic curves for defect relations},
J. Anal. Math. \textbf{74} (1998), 307--323.
\doi{10.1007/BF02819454}.

\bibitem[Ere02]{Ere02}
A.~Eremenko,
\emph{Value distribution and potential theory},
Proceedings of the International Congress of Mathematicians
(Beijing, 2002), vol.~II, Higher Education Press, Beijing, 2002,
pp.~681--690.

\bibitem[Ere15]{Ere15}
A.~Eremenko,
\emph{Holomorphic curves with few inflection points},
unpublished problem note, April~4, 2015, 2~pp.
\url{https://www.math.purdue.edu/~eremenko/dvi/wronsk.pdf}.

\bibitem[ES91]{ES91}
A.~Eremenko and M.~Sodin,
\emph{Meromorphic functions of finite order with a maximum sum of defects},
Teor. Funktsii Funktsional. Anal.\ i Prilozhen. No.~55 (1991), 84--95
(Russian); English translation, \emph{Meromorphic functions of finite
order with maximal deficiency sum}, J. Soviet Math. \textbf{59} (1992),
no.~1, 643--651.
\doi{10.1007/BF01102487}.

\bibitem[ES92]{ES92}
A.~Eremenko and M.~Sodin,
\emph{Distribution of values of meromorphic functions and meromorphic
curves from the standpoint of potential theory},
St. Petersburg Math. J. \textbf{3} (1992), no.~1, 109--136;
Russian original, Algebra i Analiz \textbf{3} (1991), no.~1, 131--164.

\bibitem[Fol99]{Fol99}
G.~B.~Folland,
\emph{Real analysis: Modern techniques and their applications},
2nd ed., Wiley-Interscience, New York, 1999.

\bibitem[Fre53]{Fre53}
M.~Frei,
\emph{Sur l'ordre des solutions enti\`eres d'une \'equation diff\'erentielle
lin\'eaire},
C. R. Acad. Sci. Paris \textbf{236} (1953), 38--40.

\bibitem[GH04]{GH04}
G.~G.~Gundersen and W.~K.~Hayman,
\emph{The strength of Cartan's version of Nevanlinna theory},
Bull. London Math. Soc. \textbf{36} (2004), no.~4, 433--454.
\doi{10.1112/S0024609304003418}.

\bibitem[Hay64]{Hay64}
W.~K.~Hayman,
\emph{Meromorphic functions},
Clarendon Press, Oxford, 1964.

\bibitem[Hio55]{Hio55} 
K.-L.~Hiong,
\emph{Sur les fonctions holomorphes dont les d\'eriv\'ees admettent
une valeur exceptionnelle},
Ann. Sci. \'Ecole Norm. Sup. (3) \textbf{72} (1955), no.~2, 165--197.
\doi{10.24033/asens.1035}.

\bibitem[Hor03]{Hor03}
L.~H\"ormander,
\emph{The analysis of linear partial differential operators.~I:
Distribution theory and Fourier analysis},
Classics in Mathematics, Springer-Verlag, Berlin, 2003,
reprint of the second (1990) edition.
\doi{10.1007/978-3-642-61497-2}.

\bibitem[KP26]{KP26}
S.~N.~Karp and K.~Purbhoo,
\emph{Universal Pl\"ucker coordinates for the Wronski map and positivity
in real Schubert calculus},
J. Amer. Math. Soc., to appear.
\doi{10.1090/jams/1087}.


\bibitem[Li11]{Li11}
B.~Q.~Li,
\emph{A logarithmic derivative lemma in several complex variables and its applications},
Trans. Amer. Math. Soc. \textbf{363} (2011), no.~12, 6257--6267.
\doi{10.1090/S0002-9947-2011-05226-8}.

\bibitem[MTV13]{MTV13}
E.~Mukhin, V.~Tarasov, and A.~Varchenko,
\emph{Bethe subalgebras of the group algebra of the symmetric group},
Transform. Groups \textbf{18} (2013), no.~3, 767--801.
\doi{10.1007/s00031-013-9232-y}.

\bibitem[Nev30]{Nev30}
F.~Nevanlinna,
\emph{\"Uber eine Klasse meromorpher Funktionen},
Septi\`eme Congr\`es des Math\'ematiciens Scandinaves
(Oslo, 1929), A.~W.~Br\o ggers Boktrykkeri, Oslo, 1930, pp.~81--83.

\bibitem[Noc83]{Noc83}
E.~I.~Nochka,
\emph{On the theory of meromorphic functions},
Dokl. Akad. Nauk SSSR \textbf{269} (1983), no.~3, 547--552 (Russian);
English translation under the title \emph{On the theory of meromorphic curves},
Soviet Math. Dokl. \textbf{27} (1983), no.~2, 377--381.

\bibitem[Pet84]{Pet84}
V.~P.~Petrenko,
\emph{Entire curves},
Vyshcha Shkola, Kharkov, 1984 (Russian).

\bibitem[Pfl46]{Pfl46}
A.~Pfluger,
\emph{Zur Defektrelation ganzer Funktionen endlicher Ordnung},
Comment. Math. Helv. \textbf{19} (1946), 91--104.
\doi{10.1007/BF02565950}.

\bibitem[Pur23]{Pur23}
K.~Purbhoo,
\emph{An identity in the Bethe subalgebra of $\C[\mathfrak S_n]$},
Proc. Lond. Math. Soc. (3) \textbf{127} (2023), no.~5, 1247--1267.
\doi{10.1112/plms.12560}.

\bibitem[Ru21]{Ru21}
M.~Ru,
\emph{Nevanlinna theory and its relation to Diophantine approximation},
2nd ed., World Scientific Publishing Co., Hackensack, NJ, 2021.
\doi{10.1142/12188}.

\bibitem[She85]{She85}
D.~F.~Shea,
\emph{On the frequency of multiple values of a meromorphic function of
small order},
Michigan Math. J. \textbf{32} (1985), no.~1, 109--116.
\doi{10.1307/mmj/1029003137}.

\bibitem[Voj97]{Voj97}
P.~Vojta,
\emph{On Cartan's theorem and Cartan's conjecture},
Amer. J. Math. \textbf{119} (1997), no.~1, 1--17.
\doi{10.1353/ajm.1997.0009}.

\bibitem[Was65]{Was65}
W.~Wasow,
\emph{Asymptotic expansions for ordinary differential equations},
John Wiley \& Sons, New York, 1965.

\bibitem[Wei69]{Wei69}
A.~Weitsman,
\emph{Meromorphic functions with maximal deficiency sum and a conjecture
of F.~Nevanlinna},
Acta Math. \textbf{123} (1969), 115--139.
\doi{10.1007/BF02392387}.

\end{thebibliography}
\end{document}